\documentclass[reqno,a4paper, 11pt]{amsart}

\usepackage[a4paper=true,pdfpagelabels]{hyperref}
\usepackage{color}
\usepackage{latexsym}
\usepackage{mathabx}
\usepackage{orcidlink}
\usepackage{amsmath}
\usepackage{tikz}
\usepackage{amssymb}

\newtheorem{theorem}{Theorem}[section]
\newtheorem{lemma}[theorem]{Lemma}
\newtheorem{corollary}[theorem]{Corollary}
\newtheorem{proposition}[theorem]{Proposition}
\newtheorem{lettertheorem}{Theorem}
\newtheorem{letterlemma}[lettertheorem]{Lemma}

\numberwithin{equation}{section}

\newcommand{\set}[1]{\left\{#1\right\}}
\newcommand{\abs}[1]{\left | #1\right |}
\newcommand{\nm}[1]{\left \| #1 \right \|}
\newcommand{\D}{\mathbb{D}}
\newcommand{\N}{\mathbb{N}}

\newcommand{\C}{\mathbb{C}}
\newcommand{\T}{\mathbb{T}}

\renewcommand{\H}{\mathcal{H}}
\renewcommand{\a}{\alpha}
\newcommand{\U}{\mathcal{U}}
\newcommand{\B}{\mathcal{B}}
\newcommand{\e}{\varepsilon}
\renewcommand{\phi}{\varphi}

\newenvironment{Prf}{\noindent{\emph{Proof of}}}
{\hfill$\Box$ }
\newcommand{\MaD}[3]{M_{#1}(D^{#2}_{#3})}
\newcommand{\maD}[3]{m_{#1}(D^{#2}_{#3})}
\title{Maximal subspaces of strong continuity for composition semigroups}
\begin{document}
\author[N. Chalmoukis]{Nikolaos Chalmoukis \orcidlink{ 0000-0001-5210-8206 }}
	\address{Dipartimento di Matematica e Applicazioni, Universit\'a degli studi di Milano Bicocca, via Roberto Cozzi, 55 20125, Milano, Italy}
	\email{nikolaos.chalmoukis@unimib.it}

\thanks{The first author is a member of the INdAM group GNAMPA and partially supported by the grant "INdAM-GNAMPA Project", CUP E53C25002010001, "Transferring Harmonic Analysis between Discrete Structures and Manifolds" and by the research grant "Yields of the ubiquity and the geometry of inner
functions (YoungInFun)", PID2024-160326NA-I00. The research of the second author is supported
in part by Ministerio de Ciencia e Innovación, Spain, project PID2022-136619NB-I00; La Junta de Andalucía, project FQM210.}
\author[A. moreno]{\'Alvaro Miguel Moreno \orcidlink{0009-0003-7220-3480}}
	\address{Departamento de Analisis Matem\'atico, Universidad de M\'alaga, Campus de Teatinos, 
		29071 Malaga, Spain}
	\email{alvarommorenolopez@uma.es}
	\maketitle
\begin{abstract}
 Let $(\varphi_t)_{t\geq 0} $ a semigroup of holomorphic self-maps of the unit disk and $C_t f = f \circ \varphi_t $ the semigroup of composition operators which corresponds to $\varphi_t. $ 
 Given a non-separable Banach space of analytic functions $X $ we study the properties of the maximal subspace of $X $ on which the semigroup $C_t $ is strongly continuous.
 In particular when $X $ contains the polynomials an interesting question is for which semigroups the maximal subspace of strong continuity coincides with the norm closure of the polynomials.

 This problem has been investigated in several function spaces including $BMOA$, $BMOA_p $, the Bloch space, $Q_s $ space and analytic Morrey spaces. However, in most cases only partial results are available. 
 
We offer a unified approach to this problem which encompasses all of the above spaces as particular examples. Moreover, we completely characterize the semigroups for which the maximal subspace of strong continuity coincides with the norm-closure of the polynomials in the space, giving therefore sharp versions of a number of results in the literature.
\end{abstract}
\section{Introduction} 

A semigroup of operators over a Fr\'echet space $X $ is a one-parameter family $(T_t)_{t\geq 0} $ of bounded operators on $X $ which satisfies the semigroup law
\[ T_{t+s}=T_t T_s, \,\, \text{for all} \,\, t,s \geq 0 \]
and $T_0 = I_X $, the identity operator on $X. $ This is merely an algebraic relation between the operators and so, usually, one has to impose an analytic condition as well. 
If $X $ is a Banach space and $ \lim_{t\to 0^+}\Vert T_t x - x \Vert = 0  $  for every $x\in X $ we say that $(T_t)_{t\geq 0} $ is a strongly continuous semigroup of operators.

Semigroups of operators are more often than not associated to flows generated by dynamical systems. In the realm of complex analysis in the unit disk $\mathbb{D} $, one such dynamical system which has been extensively studied is described by the following Cauchy problem
\begin{equation*}
	\begin{cases}
		\frac{dx(t)}{dt} = G(x(t)), \\
		x(0)=z
	\end{cases}
\end{equation*}

\noindent where $G: \mathbb{D} \to \mathbb{C} $ is a holomorphic function such that the above problem has a (unique) solution $x(t) = x^z(t) = :\varphi_t(z) $  for all $t\geq 0 $ and all initial data $z\in \mathbb{D}. $  
The one-parameter family of functions $(\varphi_t)_{t\geq 0} $ is called a semigroup of holomorphic self-maps of the unit disk and $G $ is called the infinitesimal generator of the semigroup.
It follows then that the functions  $ \varphi_t   $ are univalent self-maps of the unit disk and they satisfy the semigroup property $\varphi_{t+s} = \varphi_t \circ \varphi_s, t,s \geq0. $ 
Consequently, the composition operators $(C_t)_{t\geq 0} $, $C_tf = f\circ \varphi_t $   acting on the Fr\'echet space $\mathcal{H}(\mathbb{D}) $ of holomorphic functions in the unit disk, form a semigroup of operators. 

The question of whether such semigroups are strongly continuous when restricted on a Banach space of analytic functions $X $ is substantially more intricate
and it was first considered by Berkson and Porta in \cite{Berkson1978} for the Hardy spaces $H^p, 1 \leq p < \infty $. 
\footnote{Since we are going to deal with a variety of spaces of holomorphic functions, in order to avoid overloading the notation in the introduction we will refer to \cite{zhu} for the exact definitions. The spaces that we are going to work with are going to be defined later.}
It is a consequence of Littlewood's subordination principle that $C_{t} $ are always bounded on $H^p $ and as it was shown by Berkson and Porta $(C_{t})_{t\geq 0} $ is a strongly continuous semigroup of operators. 
The interplay between the function theory of the univalent functions $(\varphi_t)_{t\geq 0 } $ on the one side and the theory of the semigroups of operators has been so successful that soon other authors considered the composition semigroup $(C_t)_{t\geq 0 } $ in a variety of holomorphic function spaces. 

Let us briefly describe this line of research. The interested reader can find more information in the survey article \cite{Siskakis1998}. The question of strong continuity bifurcates naturally in the case that the polynomials span a dense subspace of the given space or not. When this is the case, triangle inequality shows that it is sufficient to prove that  
\begin{equation*} 
	\lim_{t\to 0^+} \Vert \varphi_t^n - e_n \Vert = 0, \,\, \text{for all} \,\, n \in \mathbb{N}, \quad \text{and} \quad \limsup_{t\to 0^+}	\Vert C_t\Vert<+\infty.
\end{equation*}
where  $e_n(z) = z^n $ and $\Vert \cdot \Vert $ is the norm of the Banach space in question.
Siskakis \cite{SiskakisDirichlet} showed that this holds for example for the Dirichlet space. Blasco et al. \cite{BlascoEtal} extended the result in the little Bloch space $\mathcal{B}_0 $  and $VMOA$.  
Even more striking is the fact that the same holds in the disk algebra $\mathcal{A}(\D) $ as shown by Gumenyuk \cite{Gumenyuk}. 
Hence, in a large class of spaces that contain the polynomials as a dense subspace, any semigroup of holomorphic functions gives rise to a strongly continuous semigroup of composition operators. 

On the other hand, when we consider the action of the composition operators $C_t $ on spaces which are non-separable, so in particular polynomials do not span a dense subspace, the situation becomes quite delicate. 
This case has been first considered implicitly by Sarason \cite{Sarason1975}, who showed that when $ \varphi_t(z)=e^{it}z $, then for $f \in BMOA $, we have that $ f\circ \varphi_t \to f $, as $t\to 0^+ $ in $BMOA $ norm if and only if $f \in VMOA $. 
This implies that the semigroup of operators $C_tf(z)=f(e^{it}z) $ is not strongly continuous on $BMOA $, rather there exists a maximal closed subspace (in this case $VMOA $) of $BMOA $ on which $C_t $ is strongly continuous.
The same result holds if in the place of $BMOA $ and $VMOA $ one considers the Bloch space $\mathcal{B} $ and the little Bloch space $\mathcal{B}$. 
In a series of two papers \cite{BlascoEtal, Blasco2008} the authors considered the problem in a systematic basis. 


Let $X $ be a Banach space of analytic functions in the unit disk, i.e. a Banach space such that $X \subseteq \mathcal{H}(\mathbb{D}) $ and the identity $I:X \to \mathcal{H}(\mathbb{D}) $ is continuous. 
From now on $(\varphi_t)$ will denote a non-trivial semigroup of holomorphic self-maps of the unit disk and $C_t $ the corresponding semigroup of composition operators on $X $. 
Then, under the hypothesis that $ \limsup_{t\to 0^+}\Vert C_t \Vert_{X} < +\infty $ it can be readily seen that the subspace
\[ [\varphi_t,X ]:= \{ f\in X : \lim_{t\to 0^+} \Vert  f\circ \varphi_t - f\Vert=0 \}, \]
is a closed subspace of $X $, which, for obvious reasons, we will call the maximal subspace of strong continuity of $(\varphi_t) $ on $X. $  
In the case that $X=\mathcal{B} $, the Bloch space as we have already mentioned $\mathcal{B}_0 \subseteq [\varphi_t, \mathcal{B}] $, where $\mathcal{B}_0 $ is the little Bloch space. Moreover, no semigroup of composition operators is strongly continuous on $\mathcal{B} $, or in our notation 
\[ [\varphi_t, \mathcal{B}] \subsetneq \mathcal{B}. \]

Since $[\varphi_t, \mathcal{B}] $ quantifies in some sense how far is $C_t $ from being strongly continuous on the Bloch space, a first natural question would be  to determine conditions on the semigroup, preferably in terms of the infinitesimal generator $G $, so that $[\varphi_t, \mathcal{B}] $ is the smallest possible, i.e.  $[\varphi_t,\mathcal{B}] = \mathcal{B}_0. $ 

We say that $G\in \mathcal{H}(\mathbb{D}) $ satisfies the vanishing logarithmic Bloch condition if 

\begin{equation}
	\lim_{|z| \to 1^{-}} \frac{1-|z|^2}{|G(z)|}\log\frac{1}{1-|z|^2} = 0.
\end{equation}

In \cite{Blasco2008} the authors proved that the vanishing logarithmic Bloch condition on the infinitesimal generator $G $ of the semigroup $(\varphi_t) $ is sufficient for $[\varphi_t, \mathcal{B}] = \mathcal{B}_0 $, but their techniques fall short of proving the necessity of the condition. 
A different sufficient condition was also obtained for the case of the spaces $BMOA $  and $VMOA $. 

The first author and Daskalogiannis recently provided the first complete result in this setting. 

\begin{lettertheorem}\cite[Theorem 1.1]{ChalDask} Let $(\varphi_t) $ and $G $ as above. The following are equivalent 
	\begin{itemize}
		\item[(i)] $G $ satisfies the logarithmic vanishing Bloch condition. 
		\item[(ii)] $[\varphi_t, \mathcal{B}]=\mathcal{B}_0 $.
		\item[(iii)] $[\varphi_t, BMOA]=VMOA $.
	\end{itemize}	
\end{lettertheorem}

Although this provides a complete picture in the case of the Bloch space and $BMOA $, semigroups of composition operators have been investigated in a variety of spaces of this kind.
In \cite{DaskGal} for example, the authors consider the so-called $BMOA_p, 1\leq p\leq 2 $ spaces which are defined by the seminorm 
\[ \sup_{a\in \mathbb{D}} (1-|a|^2)\int_\mathbb{D} |f'(z)|^p\frac{(1-|z|^2)^{p-1}}{|1-\overline{a}z|^2}d A(z) < +\infty. \]
The closure of the polynomials in this norm is denoted by $VMOA_p $. For any holomorphic semigroups $(\varphi_t)_{t\geq 0} $ the authors prove that the maximal subspace of strong continuity satisfies 
\[ VMOA_p \subseteq [\phi_t, BMOA_p] \subsetneq BMOA_p. \]

Furthermore, in analogy with the $BMOA $ case the they prove that for $VMOA_p = [\phi_t, BMOA_p] $ to hold the following condition is sufficient; 

\[ \lim_{|I|\to 0} \frac{\Big( \log\frac{2}{|I|} \Big)^p}{|I|} \int_{S(I)}\frac{(1-|z|^2)^{p-1}}{|G(z)|^p}d A(z) = 0, \]

where $dA $ is the Lebesgue area measure,  $I\subseteq \partial \mathbb{D} $ is an open arc and $S(I) = \{z\in \mathbb{D}\setminus \{0\} : 1-|z|<|I|, z/|z| \in I \} $. 
Similar results have been also obtained for the $Q_s $ spaces by Wu and Wulan \cite{Wu2021}, and for the analytic Morrey spaces $H^{2,\lambda} $ in \cite{GalMerchanSiskakis}. 

The main limitation of the above results is that they do not provide a complete characterization of when the maximal subspace of strong continuity coincides with the closure of polynomials in the space with the exception of the results in \cite{ChalDask, Arevalo2019}. 
Furthermore, the sufficient conditions that are obtained are ad hoc for each specific example, and the relationship between them remain unclear.
Moreover, the techniques employed apply separately to each space, and thus no unifying framework emerges. 

The aim of this work is, first to give sharp results in all the above cases therefore establishing sufficient and necessary conditions for the minimality of the maximal subspace of strong continuity in each case, therefore answering completely the questions regarding the maximal subspace of strong continuity in each particular case studied in the papers \cite{DaskGal, GalMerchanSiskakis, Wu2021}. 
Furthermore, we propose a unified approach that treats all the above cases as particular examples. 
In this way our theorems automatically extend to a larger class of spaces. 
Let us try to describe the main results in a bit more detail.

For $1<p<\infty$ and $-1<s<\infty$, the weighted Dirichlet space $D^p_s$ consists of analytic functions on $\D$ such that its derivative belongs to the standard Bergman space $A^p_s$, meaning
$$
\int_\D \abs{f'(z)}^p(1-\abs{z}^2)^s dA(z)<\infty.
$$
These are Banach spaces equipped with the norm $\nm{f}_{D^p_s}=\abs{f(0)}+\nm{f'}_{A^p_s}$. In particular, $D^2_0 = \mathcal{D}$ is the classical Dirichlet space, $D^p_{p-2}=B^p$ are the classical Besov spaces and by Littlewood-Paley estimates \cite{Garnett2006} $D^2_1=H^2$ the Hardy space, and if $s>p-1$ then $D^p_s=A^p_{s-p}$.

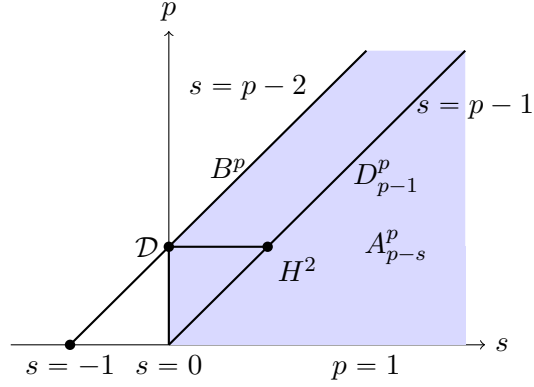
\begin{figure}[h]
\centering
\begin{tikzpicture}[scale=1.3]

\draw[->] (-1.6,0) -- (3.2,0) node[right] {$s$};
\draw[->] (0,0) -- (0,3.2) node[above] {$p$};


\fill[blue!15]
    (0,0) -- (0,1) -- (3,1) -- (3,0) -- cycle;

\fill[blue!15]
    (0,1) -- (3,1) -- (3,3) -- (2,3) -- cycle;


\node[below] at (-1,0) {$s=-1$};
\node[below] at (0,0) {$s=0$};
\node[below] at (2,0) {$p=1$};

\fill (-1,0) circle (1.5pt);
\fill (0,1) circle (1.5pt);
\fill (1,1) circle (1.5pt);

\node[left] at (0,1) {$\mathcal{D}$};
\node[below right] at (1,1) {$H^2$};

\draw[thick] (-1,0) -- (2,3);
\node[above] at (0.8,2.4) {$s=p-2$};

\draw[thick] (0,0) -- (3,3);
\node[above right] at (2.4,2.2) {$s=p-1$};

\draw[thick] (0,1) -- (1,1);

\draw[thick] (0,0) -- (0,1);

\node at (0.6,1.8) {$B^p$};
\node at (2.2,1.7) {$D^{p}_{p-1}$};
\node at (2.3,1) {$A^p_{p-s}$};

\end{tikzpicture}
\caption{The family of spaces $D^p_s $. The blue region is the admissible region of the exponents $p, s.$  }
\end{figure}

Let $1<p<\infty$, $p-2\leq s<\infty$ and $0\leq \alpha<\infty$, we define the space $\MaD{\a}{p}{s}$ as the subspace of $D^p_s$ of the functions $f$ such that
$$
\rho_{\a,p,s}(f)=\sup_{a\in \D} (1-\abs{a}^2)^{\a} \nm{f\circ\phi_a - f(a)}_{D^p_s}<\infty,
$$ 
where $\phi_a(z)=\frac{a-z}{1-\overline{a}z}$ are the automorphism of the unit disk that are also involutions. One can think of $M_\alpha(D^p_s) $ as the $\alpha $ - M\"obius invariant version of the weighted Dirichlet space $D^p_s. $ 
These spaces generalize several spaces of analytic functions such as $Q_s = \MaD{0}{2}{s}$ for $0<s<\infty$, in particular these coincide with $BMOA$ if $s=1$, and the Bloch space $\B$ if $s>1$; the BMOA-type spaces $BMOA_p=\MaD{0}{p}{p-1}$; the Morrey spaces $H^{2,\lambda}=\MaD{\frac{1-\lambda}{2}}{2}{1}$ for $0<\lambda<1$, or the growth spaces $\mathcal{A}^{-\alpha} = M_\alpha(D^p_s)  $ if $ 0<\alpha$, $1 < p < \infty$ and $s > p\a + p -1.$      
We equip these spaces with the norm
$$
\nm{f}_{\MaD{\a}{p}{s}} = \abs{f(0)}+\rho_{\a,p,s}(f).
$$ 

We also consider the ``small version'' for theses spaces
$$
\maD{\a}{p}{s} = \set{f\in\MaD{\a}{p}{s}: \lim_{\abs{a}\to 1^-}(1-\abs{a}^2)^{\a} \nm{f\circ\phi_a - f(a)}_{D^p_s}=0},
$$
for detailed info about this spaces and its basic properties see Section 2.

In general, before considering the question of strong continuity on the spaces $M_\alpha(D^p_s) $ we need to know when composition operators are well defined on $M_\alpha(D^p_s) $. Keep in mind that for $\alpha \geq \frac{s-(p-2)}{p} $ $M_\alpha(D^p_s) = D^p_s $ (see Proposition \ref{prop: classif Mad}).   

\begin{theorem}\label{thm:1.1}
 Let $1<p<\infty $, $p-2<s<\infty, 0 \leq \alpha < \frac{s-(p-2)}{p} $ and $\varphi : \mathbb{D} \to \mathbb{D} $ analytic. Let $C_\varphi $ the composition operator $C_\varphi f = f\circ \varphi. $ 
 Then
 \begin{itemize}
	 \item[(i)] If $s>p-1 $, $C_\varphi  $ is bounded on $M_\alpha(D^p_s) $;
     	\item[(ii)] If $\varphi $ is univalent and $ 2 \leq p < \infty  $ and $p-2<s \leq p-1 $, or $1<p<2 $ and $0\leq s \leq p-1 $, then $C_\varphi $ is bounded on $M_\alpha(D^p_s) $.   
 \end{itemize}
	Moreover,
$$
\nm{C_\phi}_{\MaD{\a}{p}{s}} \lesssim
\begin{cases}
	\log\!\dfrac{e}{1 - \abs{\phi(0)}}, & \text{if } \a=0, \\[1em]
	\left( \dfrac{1}{1 - \abs{\phi(0)}} \right)^\a, & \text{if } \a > 0.
\end{cases}
$$
\end{theorem}
\begin{figure}[h]
	\centering
	\begin{tikzpicture}[scale=1.3]
		
		\draw[->, dashed] (0,0) -- (3.2,0) node[right] {$s$};
		\draw[thick] (-1.6,0) -- (0,0);
		\draw[->] (0,0) -- (0,3.2) node[above] {$p$};
		
		
		\fill[blue!15]
		(0,0) -- (0,1) -- (3,1) -- (3,0) -- cycle;
		
		\fill[blue!15]
		(0,1) -- (3,1) -- (3,3) -- (2,3) -- cycle;
		
		
		\node[below] at (0,0) {$s=0, p=1$};
		
		\fill (1,1) circle (1.5pt);
		
		\node[left] at (0,1) {$\mathcal{D}$};
		\node[below right] at (0.3,0.5) {$BMOA_p$};
		
		\draw[dashed] (0,1) -- (2,3);
		\node[above] at (0.8,2.4) {$s=p-2$};
		
		\draw[thick] (0,0) -- (3,3);
		\node[above right] at (2.4,2.2) {$s=p-1$};
		
		\draw[thick] (0,1) -- (1,1);
		
		\draw[thick] (0,0) -- (0,1);
		
		\node at (0.6,1.8) {$B^p$};
		\node at (2.5,1.5) {$\mathcal{B} $ };
		\node at (0.5,0.8) {$Q_s $ };
		\node at (1.6,1) {$BMOA$};
		
	\end{tikzpicture}
	\caption{The family of spaces $M_0(D^p_s) $. The blue region is the admissible region for the parameters $p, s $. }
\end{figure}
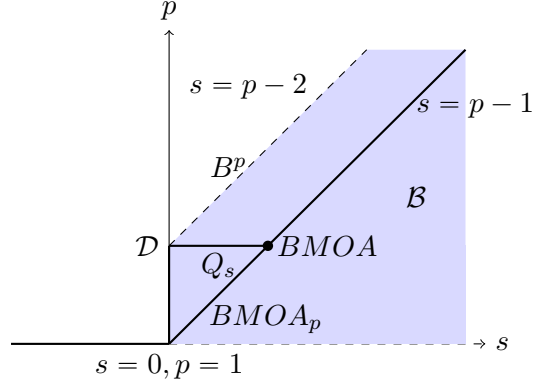

 Since all members of a semigroup are univalent functions, the corresponding composition semigroup $C_t $ is well defined in the above range of $p,s $ and $\alpha $. Moreover, the bound on the norm of the operator allows us to check that $[\phi_t,M_\a(D^p_s)]$ is a closed subspace of $M_\a(D^p_s)$. We will call this range the admissible range of parameters for brevity. Explicitly the admissible range is the set of parameters $p, s, \alpha $ such that $1 < p < \infty$, $0\leq \alpha <\frac{s-(p-2)}{p} $ and either $p\geq 2 $ and $s > p - 2 $, or $1<p<2$ and $ s \geq 0$.  
 
Next we turn to the study of the maximal subspace of strong continuity. 

\begin{theorem}\label{thm:1.2}

	Let $p,s,\alpha $ in the admissible range and $(\varphi_t) $ a semigroup of holomorphic functions. Then 
	\begin{equation*}
	m_\alpha(D^p_s)\subset	[\varphi_t, M_\alpha(D^p_s)] \subsetneq M_\alpha(D^p_s)
	\end{equation*}
	i.e. every semigroup induces a strongly continuous semigroup of composition operators on $m_{\a}(D^p_s)$, and $[\varphi_t, M_\alpha(D^p_s)]$ is a proper closed subspace of $M_\alpha(D^p_s)$ . 
\end{theorem}

Finally we come to the crux of the matter, which is to determine necessary and sufficient conditions on the infinitesimal generator in order to have that $[\varphi_t, M_\alpha(D^p_s)]=m_\alpha(D^p_s) $. 
As it turns out, the simplicity of the answer contrasts the complicated nature of the spaces $M_\alpha(D^p_s) $. We first give the answer in the case that $\alpha = 0 $, or equivalently when $M_\alpha(D^p_s) $ is the M\"obius invariant version of $D^p_s $. 

\begin{theorem}\label{thm:1.3}
	Let $p,s $ in the admissible range. The following are equivalent; 
	\begin{itemize}
		\item[(i)] $[\varphi_t, M_0(D^p_s)]=m_0(D^p_s) $,
		\item[(ii)] The infinitesimal generator $G $ satisfies the vanshing logarithmic Bloch condition. 
	\end{itemize}
\end{theorem}

In light of the above theorem we therefore see that the class of semigroups for which the maximal subspace of strong continuity is the smallest possible one does not depend on the parameters $p,s $. Therefore, such semigroups are the same for  the spaces $BMOA$, $\mathcal{B}$, $Q_s$, $BMOA_p$ which correspond to particular choices of the parameters. 
Hence the problems studied in \cite{Wu2021,ChalDask,DaskGal} have one and the same answer. 

The situation is different when $\alpha>0 $. 

\begin{theorem} \label{thm:1.4}
	Let $p,s,\alpha>0$ in the admissible range. The following are equivalent  	
	\begin{itemize}
		\item[(i)] $[\varphi_t, M_\alpha(D^p_s)]=m_\alpha(D^p_s) $,
		\item[(ii)] The infinitesimal generator $G $ satisfies the vanshing  Bloch condition; 
			\begin{equation*}
				\lim_{|z|\to 1^{-}} \frac{1-|z|^2}{|G(z)|} = 0.
			\end{equation*}
	\end{itemize}

\end{theorem}

Again we note the independence of the vanishing Bloch condition from the parameters. 
One should compare this result with Theorem 16 of \cite{Arevalo2019}, where the authors prove that the same vanishing Bloch condition is necessary and sufficient for the analogous problem on the maximal subspace of strong continuity in mean growth spaces. 
It is interesting that the techniques employed in \cite{Arevalo2019} seem to be quite distant from the ones used in this work.

The proof of Theorems \ref{thm:1.1} and $\ref{thm:1.2} $ is pretty standard and it follows after some initial considerations on the spaces $M_\alpha(D^p_s) $ in Section \ref{sec:CompositionOperators}.  
In fact, it turns out that the spaces $M_\alpha(D^p_s) $ coincide with a class of spaces introduced by Zhao \cite{zhaosurvey}. 
This identification will expedite the development of the basic theory of the spaces $M_\alpha(D^p_s). $ 

On the other hand Theorems \ref{thm:1.3} and \ref{thm:1.4} require more work. 
The basic philosophy stems from the work of Blasco et al. \cite{BlascoEtal}. 
That is, given a semigroup $(\phi_t)_{t\geq 0} $ there exists a holomorphic function $\gamma $ defined in terms of the infinitesimal operator $G $ such that (see Section \ref{sec:Notation})

\begin{equation}\label{eq:Blasco} [\phi_t, M_\alpha(D^p_s)] = \overline{M_\alpha(D^p_s) \cap (T_\gamma(M_\alpha(D^p_s)) \oplus \mathbb{C}) },\end{equation}
where $T_\gamma $ is the generalized Volterra operator  
\[ T_\gamma f(z) = \int_0^z f(\zeta)\gamma'(\zeta) d\zeta. \]

This equality, transforms the problem of determining the maximal subspace of strong continuity in a problem of determining certain properties of the image of the operator $T_\gamma. $ 
The main difficulties that one faces then is that the boundedness and compactness of $T_\gamma $ on the spaces $M_\alpha(D^p_s) $  is not fully understood, in fact it is an open problem \cite{PauZhao} to determine the class of symbols which define bounded or, respectively, compact generalized Volterra operators on the spaces $M_\alpha(D^p_s) $.
It should be mentioned that this was not the case in \cite{ChalDask}, where we have complete knowledge of the mapping properties of $T_\gamma $ on $BMOA $ and $\mathcal{B}. $ 
For this reason it is important to make non-trivial use of the extra properties of the symbol $\gamma $ in this case. 
Essentially $\gamma $ behaves like a univalent function which is then crucially exploited in Theorem~\ref{thm:1.3}
 and Theorem\ref{thm:1.4}~.

The second difficulty comes from the fact that in general $T_\gamma $ could be unbounded, and therefore part of the information contained in the image of $T_\gamma $ is lost in equation \eqref{eq:Blasco}. 
This is the most technical part of the proof and it entails a careful construction of a function in the image of $T_\gamma $ which 
acts as a ``witness'' for the non-compactness of $T_\gamma $. 
The general scheme for $\alpha = 0 $  is similar to the one in \cite[Proposition 3.6]{ChalDask}, while for $\alpha>0 $ one has modify the technique. 
This is carried out in Proposition~\ref{propo: construction function general}. 

The rest of the paper is organized as follows. Section 2 introduces all the necessary notation and definitions on analytic semigroups. Section 3 contains basic properties of $M_\a(D^p_s)$ and the study of composition operator on these spaces. Section 4 presents several useful lemmas about univalent functions in the spirit of the classical result by Pommerenke \cite{Pommerenke1978} about $BMOA$ and $\B$. Section 5 is devoted to the study of the generalized Volterra operator acting on $M_\a(D^p_s)$. Finally, in section 6 we present the main results about the maximal subspace of continuity and its proofs. 

Finally, we introduce the following notation, for two non-negative functions $a, b $ we will write $a\lesssim b$ if there exists a constant
$C>0$ such that $a\leq Cb$, and $a\gtrsim b$ is understood in an analogous manner. In
particular, if $a \lesssim b$ and $a \gtrsim b$, then we write $a\asymp b$ and say that $a$ and $b$ are comparable.
\section{Notation and definitions}\label{sec:Notation}
Let $\H(\D)$ denote the space of analytic functions in the unit disk $\D=\set{z\in\C : \abs{z}<1}$ of
the complex plane, and $\U$ be the space of all univalent functions on $\D$. 

We will first give an equivalent definition of a semigroup of holomorphic functions.
A semigroup on $\D$, $(\phi_t)_{t\geq 0}$ is defined as a family $\set{\phi_t: t\geq 0}$ of analytic self-maps on the unit disk $\phi_t:\D\to\D$ such that
\begin{enumerate}
	\item $\phi_0(z)\equiv z$,
	\item $\phi_t \circ \phi_s = \phi_{s+t},$ $t,s\geq0$,
	\item $\phi_t(z)\to z$ uniformly on compact subsets of $\D$, as $t\to0^+$.
\end{enumerate}
It turns out that if $(\phi_t)$ is a semigroup then each $\phi_t\in\U$, see \cite[Theorem 8.1.17]{BracContDiazBook}. In addition, for all members of a semigroup $(\phi_t)$ (not hyperbolic rotations), there exists a common "fixed point" $\tau\in\overline{\D}$ for which
$$
\lim_{t\to\infty} \phi_t(z)=\tau,\quad z\in\D,
$$
this point is called the Denjoy-Wolff's point of $(\phi_t)$. The concept of Denjoy–Wolff point of a
semigroup plays a key role in the semigroup theory, and we can classify semigroups with
respect to their Denjoy–Wolff point, according to whether $\tau \in \mathbb{D} $ or $\tau \in \partial \mathbb{D} $. 
In the first case we say that the semigroup is elliptic, while in the latter we simply say that the semigroup is  non-elliptic. 

If $(\phi_t)$ is a semigroup, then the limit
$$
G(z)=\lim_{t\to0} \frac{\phi_t(z)-z}{t},
$$
exists uniformly on compact subsets of $\D$. The holomorphic function $G $  is called the infinitesimal generator of the semigroup. This function characterizes the semigroup in the sense that
$$
G(\phi_t(z))= \frac{\partial\phi_t(z)}{\partial t} = G(z)\frac{\partial \phi_t(z)}{\partial z},\quad z\in\D, \, t\geq 0.
$$
By Berkson-Porta formula \cite[Theorem 10.1.10]{BracContDiazBook} we can represent $G$ in terms of the Denjoy-Wolff's point $\tau\in\overline{\mathbb{D}}$ of the semigroup
$$
G(z)=(\overline{\tau}z-1)(z-\tau)p(z),\quad z\in\D,
$$
where $p\in\H(\D)$ with $\Re{p(z)}\geq 0$, $z\in\D$. Conversely, every function
of this form is the infinitesimal generator of a unique holomorphic semigroup. 

A geometric description of a semigroup $(\phi_t)$ is provided by the so-called Koening's map, a conformal map which conjugate a given $(\phi_t)$ semigroup to a model semigroup. 

When $(\phi_t)$ is an elliptic semigroup with Denjoy-Wolff point $\tau\in\D$, $h$ is the unique conformal map such that $h(\tau)=0$, $h'(\tau)=1$ and 
$$
h(\phi_t(z))= e^{-\lambda t}h(z),\quad z\in\D, \,t\geq 0,
$$
where $\lambda\in \C\setminus \{0\}$ with $\Re(\lambda)\geq 0$ such that
$$
\phi_t'(\tau)=e^{-\lambda t},\quad t\geq 0.
$$

In addition, $\frac{h'(z)}{h(z)}=\frac{-\lambda}{G(z)}$.
If $(\phi_t)$ is non-elliptic, $h$ is the unique conformal map such that $h(0)=0$ and
$$
h(\phi_t(z))=h(z)+it,\quad z\in\D,\, t\geq 0.
$$
In this case, we have $h'(z)=\frac{i}{G(z)}$.
For a semigroup $(\phi_t)$ with infinitesimal generator $G$ and Denjoy-Wolff's point $\tau\in\overline{\D}$, we consider the function $\gamma:\D\to \C$, which we will call the $\gamma-$symbol of $(\phi_t)$ (see \cite[Definition 4]{BlascoEtal}). This function is defined as follows: if $\tau\in\D$, then 
$$
\gamma(z)= \int_{\tau}^{z}\frac{\zeta-\tau}{G(\zeta)}d\zeta,
$$
while if $\tau\in\T$, then 
$$
\gamma(z)=\int_{0}^{z}\frac{i}{G(\zeta)} d\zeta.
$$
In the case, where $\tau\in\T$ then $\gamma=h$, the Koening's map. 

Let $(\phi_t)$ be a semigroup on $\D$, we consider the associated semigroup of composition operators
$$
C_tf=f\circ\phi_t,\quad f\in\H(\D),
$$ 
The maximal subspace of strong continuity, for a semigroup can also be described in terms of the infinitesimal generator $G$. If $X$ contains the constants functions and $\limsup_{t \to 0^+}\nm{C_t}_{X}<\infty$, then by \cite[Theorem 1]{BlascoEtal}
\begin{equation}\label{eq: maximal subs 1}
	[\phi_t,X] = \overline{\set{f\in X : Gf'\in X}}
\end{equation}
This definition shows a relation between the maximal subspace of continuity and the generalized Volterra operator. For a symbol $g\in\H(\D)$ the associated Volterra operator is
$$
T_gf(z)=\int_0^z f(\zeta)g'(\zeta)d\zeta,\quad f\in\H(\D),\,z\in\D.
$$

It happens that, if $X$ is a Banach space of analytic functions on $\D$, under mild additional assumptions on $X$, by \cite[Proposition 2]{BlascoEtal}
\begin{equation*}
	[\phi_t,X] = \overline{X\cap(T_{\gamma}(X)\oplus\C)},
\end{equation*}
here we refer by $\C$ to the set of all constant functions.
\section{Basics on $\MaD{\a}{p}{s}$ and Composition operators}\label{sec:CompositionOperators}

In this section we will study some basic properties of $M_\alpha(D^p_s) $ and also obtain a description of $\MaD{\a}{p}{s}$ as the $F(p,q,s)$ spaces introduced by Zhao \cite{zhao96}. For $0<p<\infty$, $-2<q<\infty$ and $0<s<\infty$, the spaces $F(p,q,s)$ consists in $f\in\H(\D)$ such that
$$
\sup_{a\in\D} \int_\D \abs{f'(z)}^p(1-\abs{z}^2)^q(1-\abs{\phi_a(z)}^2)^s dA(z)<\infty.
$$
These spaces has been deeply studied in the last 25 years, contain only constants if $q+s\leq -1$, and if $q+s>-1$ they are non-separable Banach spaces equipped with the expected norm and the closure of the polynomials is 
$$
F_0(p,q,s) = \set{f\in F(p,q,s): \lim_{\abs{a}\to 1^-}\int_\D \abs{f'(z)}^p(1-\abs{z}^2)^q(1-\abs{\phi_a(z)}^2)^s dA(z)=0}.
$$ 
They are also always contained in the weighted Bloch space
$$
B^\a = \set{f\in\H(\D): \sup_{z\in\D} (1-\abs{z}^2)^\a\abs{f'(z)}<\infty},
$$
with $\a=\frac{q+2}{p}$, in fact this embedding turns into an equality if and only if $s> 1$. These and other key properties of $F(p,q,s)$ will be helpful for our purpose, see \cite{zhaosurvey}. 

The following well known lemma will play a key role throughout our work.
\begin{letterlemma}\label{lemmalet: equiv norm Disk-Box}
	Let $\mu$ be a positive Borel measure on $\D$. Then
	\begin{enumerate}
		\item[(i)] For each $\lambda,s>0$, there exists positive constants $C_1=C_1(\lambda,s)>0$ and $C_2=C_2(\lambda,s)>0$ such that
		$$
		C_1 \sup_{a\in\D}\frac{\mu(S(a))}{(1-\abs{a}^2)^s} \leq \sup_{a\in\D}\int_{\D}\frac{(1-\abs{a}^2)^\lambda}{\abs{1-\overline{a}z}^{\lambda+s}}d\mu(z) \leq C_2\sup_{a\in\D}\frac{\mu(S(a))}{(1-\abs{a}^2)^s};
		$$
		\item[(ii)] For each $\lambda,s>0$, $\lim_{\abs{a}\to 1^-}\int_{\D}\frac{(1-\abs{a}^2)^\lambda}{\abs{1-\overline{a}z}^{\lambda+s}}d\mu(z)=0$ if and only if 
		$$\lim_{\abs{a}\to 1^-} \frac{\mu(S(a))}{(1-\abs{a}^2)^s}=0.
		$$
	\end{enumerate}
\end{letterlemma}
\begin{proposition}\label{prop: classif Mad}
	Let $1<p<\infty$, $p-2\leq s<\infty$ and $0\leq \a <\infty$, the following holds:
	\begin{enumerate}
		\item[(i)] If $\a \geq \frac{s-(p-2)}{p}$, then $$\MaD{\a}{p}{s}=D^p_s,
		$$
		with norm equivalence;
		
		\item[(ii)] If $s>p-2$ and $0\leq \a < \frac{s-(p-2)}{p}$, then 
		$$
		\MaD{\a}{p}{s} = F(p,p(\a+1)-2, s - (p(\a+1)-2)).
		$$
		Moreover, for every $\beta>0$,
		\begin{equation}\label{eq: equiv norm maD}
			\begin{split}
		\nm{f}_{\MaD{\a}{p}{s}}&\asymp \abs{f(0)} + \left (\sup_{a\in\D}\frac{1}{(1-\abs{a}^2)^{s-(p(\a+1)-2)}}\int_{S(a)}\abs{f'(z)}^p(1-\abs{z}^2)^sdA(z) \right )^{\frac{1}{p}} \\
		&\asymp \abs{f(0)} + \left (\sup_{a\in\D}\int_\D\abs{f'(z)}^p\frac{(1-\abs{z}^2)^s(1-\abs{a}^2)^\beta}{\abs{1-\overline{a}z}^{s-(p(\a+1)-2)+\beta}}dA(z) \right )^{\frac{1}{p}};
			\end{split}
		\end{equation}
		\item[(iii)] If $s>p-2$ and $0\leq \a < \frac{s-(p-2)}{p}$, then 
		$$
		\maD{\a}{p}{s} = F_0(p,p(\a+1)-2, s - (p(\a+1)-2)).
		$$
	\end{enumerate}
\end{proposition}
\begin{proof}
	For (i) we just need to check that, if $s\geq p-2$ and $\a \geq\frac{s-(p-2)}{p}$ then
	$$
	(1-\abs{a}^2)^{p\a}\nm{f\circ\phi_a - f(a)}_{D^p_s}^p \lesssim \int_\D \abs{f'(z)}^p(1-\abs{z}^2)^s dA(z),\quad \forall a\in\D,\, f\in\H(\D).
	$$
Now if $a\in\D$ and $f\in D^p_s$, by the change of variables $w=\phi_a(z)$
	\begin{multline*}
			(1-\abs{a}^2)^{p\a}\nm{f\circ\phi_a - f(a)}_{D^p_s}^p \leq (1-\abs{a}^2)^{s-(p-2)}\nm{f\circ\phi_a - f(a)}_{D^p_s}^p \\
			\begin{aligned}
			&=(1-\abs{a}^2)^{s-(p-2)} \int_\D \abs{f'(\phi_a(z))}^p \abs{\phi_a'(z)}^2 \abs{\phi_a'(z)}^{p-2}(1-\abs{z}^2)^sdA(z) \\
			&= (1-\abs{a}^2)^{s-(p-2)} \int_\D \abs{f'(\phi_a(z))}^p \abs{\phi_a'(z)}^2 (1-\abs{\phi_a(z)}^2)^{p-2}(1-\abs{z}^2)^{s-(p-2)}dA(z) \\
			&= \int_\mathbb{D} |f'(\phi_a(z))|^p |\phi'_a(z)|^2(1-|\phi_a(z)|^2)^s |1-\overline{a}z|^{2(s-(p-2))}d A(z) \\
			&\leq 4^{s-(p-2)}\int_\D \abs{f'(w)}^p  (1-\abs{w}^2)^s dA(w). 
		\end{aligned}
	\end{multline*}
	Now, to prove the equality in (ii) by Lemma~\ref{lemmalet: equiv norm Disk-Box} (i) and the definition of $F(p,q,s)$ is enough checking that
	$$
	\sup_{a\in\D} (1-\abs{a}^2)^{p\a}\nm{f\circ\phi_a - f(a)}_{D^p_s}^p \asymp\sup_{a\in\D} \frac{1}{(1-\abs{a}^2)^{s-(p(\a+1)-2)}}\int_{S(a)}\abs{f'(z)}^p(1-\abs{z}^2)^s dA(z),
	$$
	first we notice that, working in the same way as previous case, for every $a\in\D$ 
	\begin{equation*}
		\begin{split}
			\nm{f\circ\phi_a - f(a)}_{D^p_s}^p &=  \int_\D \abs{f'(\phi_a(z))}^p \abs{\phi_a'(z)}^2 (1-\abs{\phi_a(z)}^2)^{p-2}(1-\abs{z}^2)^{s-(p-2)}dA(z) \\
			&= \int_\D \abs{f'(w)}^p(1-\abs{w}^2)^{p-2}(1-\abs{\phi_a(w)}^2)^{s-(p-2)}dA(w) \\
			&= \int_\D \abs{f'(w)}^p(1-\abs{w}^2)^{s} \frac{(1-\abs{a}^2)^{s-(p-2)}}{\abs{1-\overline{a}w}^{2(s-(p-2))}}dA(w),
		\end{split}
	\end{equation*}
	
	then by applying Lemma~\ref{lemmalet: equiv norm Disk-Box} (i) with $\lambda = s + p\a - (p-2)$

	\begin{equation*}
		\begin{split}
		\sup_{a\in\D} (1-\abs{a}^2)^{p\a}\nm{f\circ\phi_a - f(a)}_{D^p_s}^p &= \sup_{a\in\D}\int_\D \abs{f'(w)}^p(1-\abs{w}^2)^s \frac{(1-\abs{a}^2)^{s-(p-2)+p\a}}{\abs{1-\overline{a}w}^{2(s-(p-2))}}dA(w) \\
		&\asymp\sup_{a\in\D} \frac{1}{(1-\abs{a}^2)^{s-(p(\a+1)-2)}}\int_{S(a)}\abs{f'(z)}^p(1-\abs{z}^2)^s dA(z).
		\end{split}
	\end{equation*}

	The proof for (iii) follows from the exact same argument applying Lemma~\ref{lemmalet: equiv norm Disk-Box} (ii).
\end{proof}

Throughout this work we will be changing between the three equivalent norms in \eqref{eq: equiv norm maD} depending in the context, to simplify the notation, we will always write $\nm{f}_{\MaD{\a}{p}{s}}$ to refer to any of these norms. 

As a byproduct, we obtain the following result about the separability of the space.
\begin{corollary}\label{coro: density polynomials}
	Let $1<p<\infty$, $p-2\leq s<\infty$ and $0\leq \a <\infty$, then
	\begin{enumerate}
		\item[(i)] If $\a\geq \frac{s-(p-2)}{p}$, $\MaD{\a}{p}{s}$ is separable and the polynomials are a dense subset;
		\item[(ii)] If $s>p-2$ and $0\leq \a < \frac{s-(p-2)}{p}$, $\MaD{\a}{p}{s}$ is non-separable and the closure of polynomials in $\MaD{\a}{p}{s}$ coincides with $\maD{\a}{p}{s}$.
	\end{enumerate}
\end{corollary}

Given an analytic function $\phi:\D\to\D$, the composition operator associated to $\phi$ is $C_\phi (f) = f\circ \phi$, our next objective is to study the boundedness of composition operators acting on $\MaD{\a}{p}{s}$, for our purpouses in this work will be enough to consider composition operators associated to univalent symbols. 

We start by studying its boundedness on $D^p_s$. The next result is probably known to experts, for the sake of completeness we give an explicit proof. 
\begin{proposition}\label{prop: compo dirichlet}
	Let $1<p<\infty$, $p-2\leq s<\infty$ and $\phi:\D\to\D$ analytic. Then
	\begin{enumerate}
		\item[(i)] If $s>p-1$, then $C_\phi$ is bounded on $D^p_s$;
		\item[(ii)] If $2\leq p < \infty$, $p-2\leq s \leq p-1$ and $\phi\in\U$, then $C_\phi$ is bounded on $D^p_s$;
		\item[(iii)] If $1< p <2$, $0\leq s \leq p-1$ and $\phi\in\U$, then $C_\phi$ is bounded on $D^p_s$. 
	\end{enumerate}
	Moreover,
	$$
	\nm{C_\phi}_{D^p_s} \lesssim
	\begin{cases}
		\log\!\dfrac{e}{1 - \abs{\phi(0)}}, & \text{if } s=p-2, \\[1em]
		\left( \dfrac{1}{1 - \abs{\phi(0)}} \right)^{\frac{s-(p-2)}{p}}, & \text{if } s > p-2.
	\end{cases}
	$$
\end{proposition}
\begin{proof}
	(i) Let $\phi:\D\to \D$ analytic function and $s>p-1$ then $D^p_s=A^p_{s-p}$ with equivalence of norms, so by Littlewood's subordination principle \cite[Theorem 11.6]{zhu}
	$$
	\nm{f\circ}\phi_{D^p_s} \lesssim \left (\frac{1}{1-\abs{\phi(0)}} \right)^{\frac{s-(p-2)}{p}}\nm{f}_{D^p_s},\quad \forall f\in\H(\D).
	$$
	Now we deal with the remaining cases,
	let $\phi:\D\to \D$ univalent function and $f\in D^p_s$, then 
	$$
	\nm{f\circ \phi}_{D^p_s} = \abs{f\circ\phi (0)} + \left (\int_{\D} \abs{f'(\phi(z))}^p\abs{\phi'(z)}^p (1-\abs{z}^2)^s dA(z)\right )^{\frac{1}{p}}.
	$$
	The bound for the first term follows from the following standard growth estimates (see \cite[Theorem 3.1]{GirPel})
	$$
	\abs{f(z)} \lesssim
	\begin{cases}
		\log\!\dfrac{e}{1 - \abs{z}}\nm{f}_{D^p_s}, & \text{if } s=p-2, \\[1em]
		 \dfrac{\nm{f}_{D^p_s}}{(1 - \abs{z})^{\frac{s-(p-2)}{p}}}, & \text{if } s > p-2.
	\end{cases}
	$$
	applied to $z=\phi(0)$. 
	
	For the second term, we will deal each case by separated way. First if $2\leq p <\infty$ and $p-2\leq s \leq p-1$, by applying Schwarz's Lemma,  and the change of variables $w=\phi(z)$
	\begin{equation*}
		\begin{split}
			&\int_{\D} \abs{f'(\phi(z))}^p\abs{\phi'(z)}^p (1-\abs{z}^2)^s dA(z) \\ &\leq \int_\D \abs{f'(\phi(z))}^p\abs{\phi '(z)}^2 (1-\abs{\phi(z)}^2)^{p-2}(1-\abs{z}^2)^{s-(p-2)}dA(z) \\ &\lesssim \left (\frac{1}{1-\abs{\phi(0)}} \right )^{s-(p-2)} \int_\D 	\abs{f'(\phi(z))}^p\abs{\phi '(z)}^2 (1-\abs{\phi(z)}^2)^{s}dA(z) \\
			&\leq \left (\frac{1}{1-\abs{\phi(0)}} \right )^{s-(p-2)} \int_\D 	\abs{f'(w)}^p (1-\abs{w}^2)^{s}dA(z) \leq \left (\frac{1}{1-\abs{\phi(0)}} \right )^{s-(p-2)} \nm{f}_{D^p_s}^p.
		\end{split}
	\end{equation*}
Notice that we have used the inequality 

\[ (1-|z|^2)(1-|\phi(0)|^2)\leq 2(1-|\phi(z)|^2), \]
which follows by applying Scwharz's inequality to the function $\phi_{\phi(0)}\circ\phi $. 
	Now if $1< p <2$ and $0\leq s \leq p-1$, by Hölder's inequality
	\begin{equation*}
		\begin{split}
	&\int_{\D} \abs{f'(\phi(z))}^p\abs{\phi'(z)}^p (1-\abs{z}^2)^s dA(z) \leq \left (\int_{\D} \abs{f'(\phi(z))}^p\abs{\phi'(z)}^2 (1-\abs{z}^2)^s dA(z) \right)^{\frac{p}{2}} \\
	&\times \left (\int_{\D} \abs{f'(\phi(z))}^p (1-\abs{z}^2)^s dA(z) \right)^{\frac{2-p}{2}} = I(f)^{\frac{p}{2}}II(f)^{\frac{2-p}{2}}.
		\end{split}
	\end{equation*}
	For the first term, working in the same way as the previous case
	$$
	I(f) \lesssim \left (\frac{1}{1-\abs{\phi(0)}} \right )^{s}\int_\D \abs{f'(\phi(z))}^p \abs{\phi'(z)}^2(1-\abs{\phi(z)}^2)^s dA(z) \leq \left (\frac{1}{1-\abs{\phi(0)}} \right )^{s} \nm{f}_{D^p_s}^p,
	$$
	on  the other hand, $II(f)=\nm{f'\circ\phi}_{A^p_s}^p$, then by \cite[Theorem 11.6]{zhu}
	$$
	II(f)\lesssim \left (\frac{1}{1-\abs{\phi(0)}} \right )^{s+2}\nm{f'}_{A^p_s}^p \leq \left (\frac{1}{1-\abs{\phi(0)}} \right )^{s+2}\nm{f}_{D^p_s}^p.
	$$
	Then by joining both inequalities
	$$
			\int_{\D} \abs{f'(\phi(z))}^p\abs{\phi'(z)}^p (1-\abs{z}^2)^s dA(z) \lesssim \left (\frac{1}{1-\abs{\phi(0)}} \right )^{s-(p-2)} \nm{f}_{D^p_s}^p,
	$$
	concluding the proof.
\end{proof}
By Proposition~\ref{prop: classif Mad} this last result also covers the boundedness of the composition operators on $\MaD{\a}{p}{s}$ for the same range of parameters $p$ and $s$ whenever $\a\geq \frac{s-(p-2)}{p}$. Now we can deal with Theorem~\ref{thm:1.1} 

\begin{Prf}{\em{  Theorem~\ref{thm:1.1}}.}
	We deal with all cases at the same time. Let $\phi:\D\to \D$ analytic (resp. univalent), for each $a\in\D$ consider the function $\sigma_a$ defined by 
	$$
	\sigma_a(z)=\phi_{\phi(a)}\circ \phi \circ \phi_a(z),\quad z\in\D,
	$$
	it defines an analytic (resp. univalent) function with $\sigma_a(\D)\subset\D$ and $\sigma_a(0)=0$, then by Proposition~\ref{prop: compo dirichlet}
	\begin{equation}\label{eq: sigma_a}
	\nm{f\circ\sigma_a}_{D^p_s}\lesssim \nm{f}_{D^p_s},\quad \forall a\in\D, \, f\in\H(\D).
	\end{equation}
	Now, let $f\in\MaD{\a}{p}{s}$ then
	$$
	\nm{f\circ\phi}_{\MaD{\a}{p}{s}}= \abs{f(\phi(0))}+\sup_{a\in\D}(1-\abs{a}^2)^\a\nm{f\circ\phi\circ\phi_a - f\circ\phi(a)}_{D^p_s},
	$$
	for the first term, by \cite[Proposition 2.5]{zhaosurvey}
	$$
	\abs{f(\phi(0))}\lesssim
	\begin{cases}
		\log\!\dfrac{e}{1 - \abs{\phi(0)}}\nm{f}_{\MaD{\a}{p}{s}}, & \text{if } \a=0, \\[1em]
		\left( \dfrac{1}{1 - \abs{\phi(0)}} \right)^\a\nm{f}_{\MaD{\a}{p}{s}}, & \text{if } \a > 0.
	\end{cases}
	$$
	On the other hand, if $a\in\D$, by Schwarz's inequality and \eqref{eq: sigma_a}
	\begin{equation*}
		\begin{split}
			&(1-\abs{a}^2)^\a\nm{f\circ\phi\circ\phi_a - f\circ\phi(a)}_{D^p_s}=(1-\abs{a}^2)^\a\nm{f\circ\phi_{\phi(a)}\circ\sigma_a - f\circ\phi(a)}_{D^p_s} \\
			&\lesssim (1-\abs{a}^2)^\a \nm{f\circ\phi_{\phi(a)} - f(\phi(a))}_{D^p_s} \lesssim \left (\frac{1-\abs{\phi(a)}^2}{1-\abs{\phi(0)}} \right )^{\a}\nm{f\circ\phi_{\phi(a)} - f(\phi(a))}_{D^p_s} \\
			&\leq \left (\frac{1}{1-\abs{\phi(0)}}\right )^\a \nm{f}_{\MaD{\a}{p}{s}},
		\end{split}
	\end{equation*}
	concluding the proof.
\end{Prf}

To conclude this section we introduce the following test functions that will be useful throughout our work, the following result follows from \cite[Lemma 2.5]{OrtegaFabrega} (see also \cite[Theorem 4.2]{zhaosurvey}).
\begin{lemma}\label{lemma: test funct a>0}
	Let $p, s, \alpha $ in the admissible range, then the following holds:
	\begin{enumerate}
		\item[(i)] Consider or every $w\in\D$ the function
		$$
		\l_\omega(z)=\log\frac{e}{1-\overline{w}z},
		$$
		then $\sup\limits_{w\in\D} \nm{l_w}_{\MaD{0}{p}{s}}<\infty$;
		\item[(ii)] If $\a,\lambda>0$, consider for every $w\in\overline{\D}$  the function
		$$
		f_{w,\alpha,\lambda} = \frac{(1-\abs{w}^2)^\lambda}{(1-\overline{w}z)^{\a+\lambda}},
		$$
		 Then $$\sup\limits_{w\in\overline{\D}}\nm{f_{w,\alpha,\lambda}}_{\MaD{\a}{p}{s}}<\infty.$$
	\end{enumerate}
\end{lemma}
\section{Univalent functions}
\begin{lemma}\label{lemma: contains Qp}
	Let $p, s $ in the admissible range, then there exists $0<q<\infty$ such that 
	$$
	Q_q\subset \MaD{0}{p}{s},\quad \text{and} \quad Q_{q,0}\subset \maD{0}{p}{s}.
	$$
\end{lemma}
\begin{proof}
	Both embeddings follow from the same ideas, we detail here the first embedding, we divide the proof in different cases. 
	
	If $s>p-1$, let $f\in \B$ then
	\begin{equation*}
		\begin{split}
			\nm{f\circ\phi_a - f(a)}_{D^p_s}^p &\lesssim \nm{f}_{\B}^p\int_{\D} \left (\frac{\abs{\phi_a'(z)}}{1-\abs{\phi_a(z)}^2}\right )^p (1-\abs{z}^2)^s dA(z)\\
			&= \nm{f}_{\B}^p\int_\D (1-\abs{z}^2)^{s-p}dA(z)\lesssim \nm{f}_{\B}^p,
		\end{split}
	\end{equation*}
	so $f\in \MaD{0}{p}{s}$ and $Q_q\subset \MaD{0}{p}{s}$ for any $0<q<\infty$. 
	
	If $2\leq p <\infty$ and $p-2<s\leq p-1$, lets check that $Q_q\subset \MaD{0}{p}{s}$ with $q=s-(p-2)$. Let $f\in Q_q$ and $a\in\D$
	\begin{equation*}
		\begin{split}
			\int_\D \abs{f'(z)}^p (1-\abs{z}^2)^{p-2}(1-\abs{\phi_a(z)}^2)^{s-(p-2)}dA(z)&\leq \nm{f}_\B^{p-2}\int_\D \abs{f'(z)}^2(1-\abs{\phi_a(z)}^2)^{q}dA(z)\\ &\leq \nm{f}_{\B}^{p-2}\nm{f}_{Q_q}^2 \lesssim \nm{f}_{Q_q}^p,		
		\end{split}
	\end{equation*}
	and by taking supremum on $a\in\D$, $f\in\MaD{0}{p}{s}$. 
	
	If $1< p<2$ and $0<s\leq p-1$, by Hölder's inequality $\nm{f}_{D^p_s}\lesssim\nm{f}_{D^2_s}$ for every $f\in\H(\D)$, so trivially $Q_s=\MaD{0}{2}{s}\subset \MaD{0}{p}{s}$. 
	
	If $1< p<2$ and $s=0$, take $ \e < \frac{2-p}{2}$ and $q=\frac{2\e}{p}$, then by Hölder's inequality,
	\begin{equation*}
		\begin{split}
			\nm{f}_{D^p_s}^p &\leq  \int_{\D} \abs{f'(z)}^p(1-\abs{z}^2)^{\e}(1-\abs{z}^2)^{-\e}dA(z) \\
			&\leq \left (\int_\D \abs{f'(z)}^2(1-\abs{z}^2)^{\frac{2\e}{p}} \right)^{\frac{p}{2}}dA(z)\left (\int_\D (1-\abs{z}^2)^{\frac{-2\e}{2-p}} dA(z)\right)^{\frac{2-p}{2}} \\
			&\asymp \nm{f}_{D^2_q}^p,\quad \forall f\in\H(\D),
		\end{split}
	\end{equation*}
	so $Q_q=\MaD{0}{2}{q}\subset \MaD{0}{p}{s}$. 
\end{proof}
\begin{lemma}\label{lemma: univalent bloch equiv maD}
	Let $p, s $ in the admissible range. If $f\in \U$, then 
	$$
	f\in\MaD{0}{p}{s}\Leftrightarrow f\in \B \quad \text{and} \quad f\in\maD{0}{p}{s}\Leftrightarrow f\in \B_0.
	$$
\end{lemma}
\begin{proof}
	First, notice that by \cite[Corollary 3.6]{zhaosurvey}, $\MaD{0}{p}{s}\subset \B$ and $\maD{0}{p}{s}\subset \B_0$, so we just need to check that if $f\in\U\cap\B$ (resp. $\U\cap\B_0$) then $f\in\MaD{0}{p}{s}$ (resp. $f\in\maD{0}{p}{s}$). 
	
	As $f\in\U$ by \cite[Theorem 6.1]{Aluaskari1997} $f\in Q_s$ (resp. $Q_{s,0}$) for every $0<s<\infty$, and by Lemma~\ref{lemma: contains Qp} there is always a $Q_s$ (resp. $Q_{s,0}$) space inside $\MaD{0}{p}{s}$ (resp. $\maD{0}{p}{s}$), so $f\in\MaD{0}{p}{s}$ (resp. $\maD{0}{p}{s}$).
\end{proof}
It will be useful to obtain the logarithmic version for these results. 

The logarithmic Bloch space $\B_{\log}$ is defined as the space of analytic functions $f$ in $\D$ such that
$$
\sup_{z\in\D} \left (\log\frac{e}{1-\abs{z}^2}\right )(1-\abs{z}^2)\abs{f'(z)}<\infty,
$$
and $f\in\B_{\log,0}$ if
$$
\lim_{\abs{z}\to 1^-}\left (\log\frac{e}{1-\abs{z}^2}\right )(1-\abs{z}^2)\abs{f'(z)}=0.
$$
For $0<p<\infty$, $-2<q<\infty$ and $0<s<\infty$, the spaces $F_{\log}(p,q,s)$ consists in $f\in\H(\D)$ such that
$$
\sup_{a\in\D}\left (\log\frac{1}{1-\abs{a}^2}\right)^p \int_\D \abs{f'(z)}^p(1-\abs{z}^2)^q(1-\abs{\phi_a(z)}^2)^s dA(z)<\infty,
$$
and for the little-o version we say that $f\in F_{\log,0}(p,q,s)$ if
$$
	\lim_{\abs{a}\to 1^-}\left (\log\frac{1}{1-\abs{a}^2}\right)^p\int_\D \abs{f'(z)}^p(1-\abs{z}^2)^q(1-\abs{\phi_a(z)}^2)^s dA(z)=0.
$$

In particular, for $0<q<\infty$ we denote by
$$
 Q_{q,\log,0}=F_{\log,0}(2,0,q),
$$
notice that, if $q>1$ then $Q_{q,\log,0}=\B_{\log,0}$.
By standard arguments and mimicking the proof of Lemma~\ref{lemma: contains Qp} we can obtain the following result.
\begin{lemma}\label{lemma: Mad log cont Qp}
	Let $p, s$ in the admissible range, then there exists $0<q<\infty$ such that
	$$
	Q_{q,\log,0}\subset F_{\log,0}(p,p-2,s-(p-2))\subset \B_{\log,0}.
	$$
\end{lemma}
Now in the same way as \cite[Corollary 3.3]{ChalDask}, we are going to prove that the logarithmic version of \cite[Theorem 6.1]{Aluaskari1997}.  

Let $\omega$ be a strictly positive weight of the class $C^1(\D)$. We assume the following regularity condition on $\omega$:
\begin{equation}\label{eq: weight def}
	(1-\abs{z}^2)\abs{\nabla \omega(z)}\leq C_\omega \omega(z),\quad \forall z\in\D,
\end{equation}
for some $C_\omega>0$. 
\begin{lemma}\label{lemma: log qp 1}
	Let $0<s<1$, and suppose that $\omega$ is a weight which satisfies \eqref{eq: weight def} with $C_\omega < s$. Let also $f\in\H(\D)$ such that
	$$
	\abs{f'(z)}(1-\abs{z}^2)\omega(z)\leq K,\quad \forall z\in\D,
	$$
	where $K>0$. Then,
	$$
	\int_0^1 \frac{ \sup_{a\in\D,\abs{z}\leq r} (\omega(a)\abs{f\circ\phi_a(z) - f(a)})^2}{(1-r)^{1-s}}dr<+\infty.
	$$
\end{lemma}
\begin{proof}
	By following the proof of \cite[Lemma 3.1]{ChalDask}, we get for $z=re^{i\theta}$ 
	$$
	(\omega(a)\abs{f\circ\phi_a(z)-f(a)})^2 \leq K^2 \left (\frac{1+r}{1-r} \right )^{C_\omega} \left (\frac{1}{2}\log \frac{1+r}{1-r} \right )^2,
	$$
	then by applying that $1+C_\omega-s<1$,
	$$
	\int_0^1 (1-r)^s \sup_{a\in\D,\abs{z}\leq r} (\omega(a)\abs{f\circ\phi_a(z) - f(a)})^2dr \lesssim \int_0^1 \frac{1}{(1-r)^{1+C_\omega-s}}\left (\log \frac{1+r}{1-r}\right )^2 dr <\infty.
	$$
\end{proof}
\begin{proposition}\label{propo: log qp 2}
	Let $0<s<1$, $f:\D \to \C$ univalent and $\omega$ a weight as in Lemma~\ref{lemma: log qp 1}. Suppose also that
	$$
	\lim_{\abs{z}\to 1}\abs{f'(z)}(1-\abs{z}^2)\omega(z)=0.
	$$
	Then,
	$$
	\lim_{\abs{a}\to 1}\omega(a)^2\int_{\D}\abs{f'(z)}^2 (1-\abs{\phi_a(z)}^2)^sdA(z)=0.
	$$
\end{proposition}

\begin{proof}
 Let $f$ be such a function. Setting $\D_r=\overline{D(0,r)}$ and denote $f_a=f\circ\phi_a - f(a)$, by applying a change of variables and Fubini's theorem if $0<R<1$,
	\begin{multline*}
		\omega(a)^2\int_{\D}\abs{f'(w)}^2(1-\abs{\phi_a(w)})^sdA(w) \asymp \omega(a)^2\int_0^1 \frac{1}{(1-r)^{1-s}}\int_{\D_r}\abs{f'_a(z)}^2dA(z)\, dr \\
		\begin{aligned}
			&= \omega(a)^2\int_0^R \frac{1}{(1-r)^{1-s}}\int_{\D_r}\abs{f'_a(z)}^2dA(z)\, dr + \omega(a)^2\int_R^1 \frac{1}{(1-r)^{1-s}}\int_{\D_r}\abs{f'_a(z)}^2dA(z)\, dr \\ 
			&= I + II.
		\end{aligned}
	\end{multline*}
	Since $f_a$ is univalent, the inner integral is the normalized area of the image $f_a(\D_r)$, then
	$$
	\int_{\D_r}\abs{f'_a(z)}^2dA(z) \leq \sup_{\abs{z}\leq r} \abs{f_a(z)}^2,
	$$
	then by Lemma~\ref{lemma: log qp 1}, if $\e>0$ there exists $R_0<1$ such that
	$$
	II\leq \omega(a)^2\int_{R_0}^1\sup_{\abs{z}\leq r}\abs{f_a(z)}^2dr <\e,\quad \forall a\in\D.
	$$
	For the remaining term, by following the proof of \cite[Proposition 3.2]{ChalDask}
	$$
	I \leq e^{2C_\omega \delta(0,R_0)}\int_0^{R_0}\frac{1}{(1-r)^{3-s}}\int_{\D_r} \abs{f'(\phi_a(z))}^2(1-\abs{\phi_a(z)}^2)\omega(\phi_a(z))^2 dA(z)\, dr,
	$$
	where $\delta$ denotes the hyperbolic distance on the disk. 
	Now, by hypothesis we can find $R_1<1$ such that if $\abs{w}>R_1$, then
	$$
	\omega(w)^2(1-\abs{w}^2)^2\abs{f'(w)}^2\leq e^{-2C_\omega \delta(0,R_0)}(1-R_0)^{3-s}\e.
	$$
	Finally, there exists some $\delta>0$ such that $\abs{\phi_a(z)}>R_1$, if $\abs{z}\leq R_0$ and $\abs{a}>1-\delta$. Then
	$$
	I\leq \e(1-R_0)^{3-s}\int_0^{R_0}\frac{1}{(1-r)^{3-s}}dr \leq \e.
	$$
	Therefore, for each $\e>0$, there exists $\delta>0$ such that for $\abs{a}>1-\delta$
	$$
	\omega(a)^2\int_{\D}\abs{f'(z)}^2(1-\abs{\phi_a(z)}^2)^sdA(z)\leq 2\e.
	$$
\end{proof}
\begin{corollary}\label{coro: Blog univ equiv Flog}
	Let $p, s $ in the admissible range. If $f\in\U$, then
	$$
	f\in F_{\log,0}(p,p-2,s-(p-2))\Leftrightarrow f\in\B_{\log,0}.
	$$
\end{corollary}
\begin{proof}
	By Lemma~\ref{lemma: Mad log cont Qp} it is enough to check that, if $f\in\B_{\log,0}$ then $f\in Q_{s,\log,0}$ with $0<s<1$. 
	
	Let $K>0$ and consider the weight $\omega(z)=\log\frac{K}{1-\abs{z}^2}$, then if $z=re^{i\theta}$
	$$
	\frac{(1-\abs{z}^2)\abs{\nabla \omega(z)}}{\omega(z)} = \frac{2r}{\log K + \log\frac{1}{1-r}},
	$$
	and this quantity converges to 0 uniformly on $(0,1)$ when $K$ goes to infinity, so for every $0<s<1$, there exists a $K_s$ such that $\omega$ verifies the conditions of Lemma~\ref{lemma: log qp 1}, then by Proposition~\ref{propo: log qp 2}, $f\in Q_{p,\log,0}$.
\end{proof}
\section{Integration operator and related results}
In view of \eqref{eq:Blasco} to understand the behaviour of the maximal subspace of strong continuity we need to understand how the integral operator $T_g $ acts on $M_\alpha(D^p_s) $. This section is devoted to this objective.

\subsection{Boundedness and compactness}
For the case $\a=0$ the integral operator is widely understood, the following result follows from \cite[Theorem 5.1,5.2]{PauZhao}, \cite[Theorem 1.4]{Chen} and Gantmacher's theorem \cite[Theorem 5.23]{PositiveOperators}
\begin{lettertheorem}\label{thlet: tg log caract}
	Let $1<p<\infty$, $p-2<s<\infty$ and $g\in\H(\D)$. $T_g$ is bounded on $\MaD{0}{p}{s}$ if and only if $g\in F_{\log}(p,p-2,s-(p-2))$. Furthermore, the following are equivalent:
	\begin{enumerate}
		\item[(i)] $T_g$ is compact on $\MaD{0}{p}{s}$;
		\item[(ii)] $g\in F_{\log,0}(p,p-2,s-(p-2))$;
		\item[(iii)] $T_g$ is weakly compact on $\MaD{0}{p}{s}$;
		\item[(iv)] $T_g(\MaD{0}{p}{s})\subset \maD{0}{p}{s}$.
	\end{enumerate}  
\end{lettertheorem}
The case $\a>0$ is much more involved, in fact there is no complete characterization of the boundedness or compactness for general analytic symbols for the whole admissible range, only for some ranges of parameters the symbols are characterized in \cite[Theorem 5.2]{PauZhao}. 
However, for our purpose, we just need this kind of result for univalent symbols, under this assumption we are able to characterize the univalent symbols $g$ such that $T_g$ is bounded/compact on $\MaD{\a}{p}{s}$ if $\a>0$ in the admissible range, first we need some previous results.  \\
The following lemma is an immediate consequence of \cite[Theorem 4.3]{PauZhao}. 
\begin{lemma}\label{lemma: bounded Dps implies Mad}
	Let $p, s ,\a>0$ in the admissible range, and $g\in\H(\D)$. If $T_g$ is bounded on $D^p_s$, then it is also bounded on $\MaD{\a}{p}{s}$. Moreover,
	$$
	\nm{T_g}_{\MaD{\a}{p}{s}}\lesssim \nm{T_g}_{D^p_s}.
	$$
\end{lemma}

\begin{proposition}\label{propo: Tg bounded Mad suff}
	Let $p, s, \alpha>0 $ in the admissible range, and $g\in\H(\D)$. If there exists $\beta>0$ such that $g\in\MaD{0}{p}{s-\beta}$, then $T_g$ is bounded on $\MaD{\a}{p}{s}$.
	Moreover,
	$$
	\nm{T_g}_{\MaD{\a}{p}{s}}\lesssim \nm{g-g(0)}_{\MaD{0}{p}{s-\beta}}.
	$$
\end{proposition}
\begin{proof}
	Assume that there exists $\beta>0$ such that $g\in\MaD{0}{p}{s-\beta}$, by  Lemma~\ref{lemma: bounded Dps implies Mad} we just need to check that $T_g$ is bounded on $D^p_s$.
	Let $\gamma>\max\left (0,\frac{s}{p}\right )$, then by \cite[Proposition 4.27]{zhu} 
	$$
	\abs{f(z)}\lesssim \int_\D\frac{\abs{f'(w)}(1-\abs{w}^2)^\gamma}{\abs{1-\overline{w}z}^{\gamma + 1}}dA(w),\quad f\in\H(\D)
	$$
	take now $\e= \beta + 2p-2$, by Hölder's inequality
	\begin{equation*}
		\begin{split}
			\abs{f(z)}^p &\leq \int_\D \frac{\abs{f'(w)}^p(1-\abs{w}^2)^{\gamma p}}{\abs{1-\overline{w}z}^{p+\gamma p-\e}}dA(w)\left ( \int_\D \abs{1-\overline{w}z}^{-\frac{\e}{p-1}} \right )^{p-1} \\
			&\asymp \frac{1}{(1-\abs{z}^2)^{\e-2p+2}}\int_\D \frac{\abs{f'(w)}^p(1-\abs{w}^2)^{\gamma p}}{\abs{1-\overline{w}z}^{p+\gamma p-\e}}dA(w),\quad z\in\D.
		\end{split}
	\end{equation*}
	Then Fubini's theorem, if $f\in D^p_s$
	\begin{equation}\label{eq: tg univ 1}
		\begin{split}
			\nm{T_gf}_{D^p_s}^p \leq \int_{\D} \abs{f'(w)}^p(1-\abs{w}^2)^{\gamma p} \int_\D \frac{\abs{g'(z)}^p(1-\abs{z}^2)^{s-\beta}}{\abs{1-\overline{w}z}^{p+\gamma p-\e}}dA(z)\, dA(w).
		\end{split}
	\end{equation}
	
	Now denote by $I(g,w)$ to the inner integral, fixed $w\in\D\setminus\{0\}$ and consider $N(w)=\min\set{n\in\N: 2^n(1-\abs{w}^2)\geq 1}$, for every $0\leq k < N(w)$ set $S_k(w)=S\left ((1-2^k(1-\abs{w}^2))e^{i\arg{w}} \right )$, $S_{-1}(w)=\emptyset$ and $S_{N(w)}=\D$. Then $\abs{1-\overline{w}z}\asymp 2^k(1-\abs{w}^w)$, if $z\in S_k(w)\setminus S_{k-1}(w)$, so
	\begin{equation*}
		\begin{split}
			I(g,w) &= \sum_{k=0}^{N(w)} \int_{S_k(w)\setminus S_{k-1}(w)} \frac{\abs{g'(z)}^p(1-\abs{z}^2)^{s-\beta}}{\abs{1-\overline{w}z}^{p+\gamma p-\e}}dA(z) \\
			&\lesssim (1-\abs{z}^2)^{\e-p-\gamma p} \sum_{k=0}^{N(w)} \frac{1}{2^{p+\gamma p -\e}}\int_{S_k(w)} \abs{g'(z)}^p(1-\abs{z}^2)^{s-\beta}dA(z) \\
			&\lesssim \nm{g-g(0)}_{\MaD{0}{p}{s-\beta}}^p(1-\abs{z}^2)^{\e-p-\gamma p + s -\beta -(p-2)} \sum_{k=0}^{N(w)} \frac{1}{2^{p+\gamma p - \e -s +\beta+p-2}} \\
			&= \nm{g-g(0)}_{\MaD{0}{p}{s-\beta}}^p(1-\abs{z}^2)^{s-\gamma p}\sum_{k=0}^{N(w)} \frac{1}{2^{\gamma p -s}}\lesssim\nm{g-g(0)}_{\MaD{0}{p}{s-\beta}}^p(1-\abs{z}^2)^{s-\gamma p}
		\end{split}
	\end{equation*}
	by joining this with \eqref{eq: tg univ 1}
	$$
	\nm{T_gf}_{D^p_s}^p \lesssim \nm{g-g(0)}_{\MaD{0}{p}{s-\beta}}^p\int_{\D} \abs{f'(z)}^p(1-\abs{z}^2)^sdA(z)\leq \nm{g-g(0)}_{\MaD{0}{p}{s-\beta}}^p\nm{f}_{D^p_s}^p.
	$$
\end{proof}
\begin{proposition}\label{propo: Tg bounded Mad univalent}
	Let $p, s, \alpha>0 $ in the admissible range, and $g\in\U$. The following holds:
	\begin{enumerate}
		\item[(i)] If $g\in\MaD{0}{p}{s}$ then there exists $\beta>0$ such that $g\in\MaD{0}{p}{s-\beta}$;
		\item[(ii)] $T_g$ is bounded on $\MaD{\a}{p}{s}$ if and only if $g\in \MaD{0}{p}{s}$.
	\end{enumerate}
	Moreover,
	$$ 
	\nm{g-g(0)}_{\MaD{0}{p}{s}}\lesssim \nm{T_g}_{\MaD{\a}{p}{s}}\lesssim \nm{g-g(0)}_{\MaD{0}{p}{s-\beta}}.
	$$
\end{proposition}
\begin{proof}
	To check (i) if we revisit the admissible range, except for the case when $1<p<2$ and $s=0$, you can always take $\beta>0$ such that $p$ and $s'=s-\beta$ verify the admissible range, so as a consequence of Lemma~\ref{lemma: univalent bloch equiv maD}
	$$
	g\in\MaD{0}{p}{s}\Leftrightarrow g\in \B \Leftrightarrow g\in\MaD{0}{p}{s-\beta}.
	$$ 
	For the case $1< p <2$ and $s=0$, if we take $\beta<\frac{2-p}{2}$ and $\beta<\beta'<\frac{2-p}{2}$, we can reproduce the proof of Lemma~\ref{lemma: contains Qp} to check that
	$$
	Q_{\frac{2(\beta'-\beta)}{p}}\subset \MaD{0}{p}{-\beta}\subset \B,
	$$
	so by the same argument as Lemma~\ref{lemma: univalent bloch equiv maD} we get that
	$$
	g\in\MaD{0}{p}{0}\Leftrightarrow g\in \B \Leftrightarrow g\in\MaD{0}{p}{-\beta}.
	$$
	To prove the equivalence in (ii), first the sufficiency follows from (i) and Proposition~\ref{propo: Tg bounded Mad suff}. For the converse, by testing on the functions 
	$$
	f_\zeta(z)=\frac{1}{(1-\overline{\zeta}z)^\alpha},\quad z,\zeta\in\D.
	$$
	then we know $\sup_{\zeta\in\D}\nm{f_\zeta}_{\MaD{\a}{p}{s}}<\infty$, so
	\begin{equation*}
		\begin{split}
			\nm{T_g}_{\MaD{\a}{p}{s}}^p &\gtrsim \sup_{a\in\D}\frac{1}{(1-\abs{a}^2)^{s-(p(\alpha+1)-2)}}\int_{S(a)} \frac{\abs{g'(z)}^p(1-\abs{z}^2)^s}{\abs{1-\overline{\zeta}z}^{p\alpha}}dA(z) \\
			&\geq \frac{1}{(1-\abs{\zeta}^2)^{s-(p(\alpha+1)-2)}}\int_{S(\zeta)} \frac{\abs{g'(z)}^p(1-\abs{z}^2)^{s}}{\abs{1-\overline{\zeta}z}^{p\alpha}}dA(z) \\
			&\asymp \frac{1}{(1-\abs{\zeta}^2)^{s-(p-2)}}\int_{S(\zeta)} \abs{g'(z)}^p(1-\abs{z}^2)^{s}dA(z),
		\end{split}
	\end{equation*}
	and the proof ends by taking supremum on $\zeta\in\D$.
\end{proof}
Now that we have characterized the univalent symbols such that $T_g$ is bounded, we deal with the compactness, with this aim we prove the following lemma.
\begin{lemma}\label{lemma: Tg compact poli}
	Let $p, s, \alpha>0 $ in the admissible range, then the following holds:
	\begin{enumerate}
		\item[(i)] The multiplication operator
		$$
		Mf(z)=zf(z),\quad z\in\D,
		$$
		is bounded;
		\item[(ii)] The classical Volterra operator
		$$
		T f(z)=\int_0^z f(\zeta)d\zeta, 
		$$
		is compact on $\MaD{\a}{p}{s}$;
		\item[(iii)] For every polynomial $P$, then $T_P$ is compact on $\MaD{\a}{p}{s}$.
	\end{enumerate} 
\end{lemma}
\begin{proof}
	First, we notice that by Proposition~\ref{propo: Tg bounded Mad univalent} then $T$ is bounded, as a byproduct $M$ is also bounded
	\begin{equation*}
		\begin{split}
			\nm{Mf}_{\MaD{\a}{p}{s}}^p &= \sup_{a\in \D} \frac{1}{(1-\abs{a}^2)^{s-(p(\a+1)-2)}}\int_{S(a)} \abs{f(z)+zf'(z)}^p(1-\abs{z}^2)^sdA(z) \\ &\lesssim
			\nm{Tf}_{\MaD{\a}{p}{s}}^p + \nm{f}_{\MaD{\a}{p}{s}}^p \lesssim \nm{f}_{\MaD{\a}{p}{s}},
		\end{split}
	\end{equation*}
	this proves (i). \\
	Now lets check (ii). First, if $s>p-1$ the compactness of $T$ follows from \cite[Theorem 5.2]{PauZhao}, so we just need to check the case $s\leq p-1$. \\
	Let $\{f_n\}$ be a uniformly bounded sequence on $\MaD{\a}{p}{s}$, such that $\{f_n\}$ converge to zero uniformly on compact subsets of $\D$, we need to check that 
	$$
	\lim_{n\to \infty} \nm{Tf_n}_{\MaD{\a}{p}{s}} =0,
	$$ 
	to simplify the notation, let $\beta =s-(p(\a+1)-2)$, we divide the norm in two different terms
	\begin{equation*}
		\begin{split}
			\nm{Tf_n}_{\MaD{\a}{p}{s}}^p &\lesssim \abs{f_n(0)}^p\sup_{a\in\D}\int_\D \frac{(1-\abs{z}^2)^s(1-\abs{a}^2)^\beta}{\abs{1-\overline{a}z}^{2\beta}}dA(z) \\
			&+ \sup_{a\in\D}\int_\D \abs{f_n(z)-f_n(0)}^p \frac{(1-\abs{z}^2)^s(1-\abs{a}^2)^\beta}{\abs{1-\overline{a}z}^{2\beta}}dA(z) \\
			&= I(f_n)+II(f_n),
		\end{split}
	\end{equation*}
	for the first term, simply by noticing that the supremum is finite, then because $f_n $ converges pointwise to $0 $, we obtain that
	$$
	\lim_{n\to \infty} I(f_n)=0.
	$$
	For the second term, by \cite[Lemma 2.1]{BlasiPau} (notice the lemma requires $2\beta<2+s$, this is equivalent to $s<2(p-1)+2p\a$, that holds as $s\leq p-1$), then
	$$
	II(f_n) \lesssim \sup_{a\in\D} \int_\D \abs{f_n'(z)}^p \frac{(1-\abs{z}^2)^{s+p}(1-\abs{a}^2)^\beta}{\abs{1-\overline{a}z}^{2\beta}}dA(z).
	$$
	Let $\e>0$, lets check that there exists $n_0\in\N$ such that if $n\geq n_0$ then 
	$$
	II(f_n')=\sup_{a\in\D} \int_\D \abs{f_n'(z)}^p \frac{(1-\abs{z}^2)^{s+p}(1-\abs{a}^2)^\beta}{\abs{1-\overline{a}z}^{2\beta}}dA(z)<\e.
	$$
	Take $R>0$ such that $(1-R^2)^p\leq \frac{\e}{2\sup_{n\in\N} \nm{f_n}_{\MaD{\a}{p}{s}}^p}$, on the other hand $f_n'$ (because $f_n$ does )converges to zero uniformly on $\overline{D(0,R)}$, then there exists $n_0\in\N$ such that if $n\geq n_0$, then $\abs{f_n'(z)}^p< \frac{\e}{2S}$ for all $z\in\overline{D(0,R)}$, where 
	$$
	S=\sup_{a\in\D} \int_\D \frac{(1-\abs{z}^2)^{s+p}(1-\abs{a}^2)^\beta}{\abs{1-\overline{a}z}^{2\beta}}dA(z)<\infty,
	$$
	putting this together, we obtain that for every $n\geq n_0$
	\begin{equation*}
		\begin{split}
			II(f_n') &\leq  \sup_{a\in\D} \int_{\overline{D(0,R)}} \abs{f_n'(z)}^p \frac{(1-\abs{z}^2)^{s+p}(1-\abs{a}^2)^\beta}{\abs{1-\overline{a}z}^{2\beta}}dA(z)  \\
			&+ \sup_{a\in\D} \int_{\set{z\in\D: \abs{z}>R}} \abs{f_n'(z)}^p \frac{(1-\abs{z}^2)^{s+p}(1-\abs{a}^2)^\beta}{\abs{1-\overline{a}z}^{2\beta}}dA(z) \\
			&< \frac{\e}{2} + (1-R^2)^p\nm{f_n}_{\MaD{\a}{p}{s}}^p < \e.
		\end{split}
	\end{equation*}

	Then we obtain that
	$$
	\lim_{n\to\infty}\nm{Tf_n}_{\MaD{\a}{p}{s}}=0,
	$$
	so $T$ is compact on $\MaD{\a}{p}{s}$. \\
	After proving (i) and (ii), (iii) follows from \cite[Proposition 2.2 (ii)]{SiskakisZhao}. 
\end{proof}
For each $\a\geq 0$ consider the family of bounded operator 
$$\mathcal{L_\a}= \set{L_{a,\a} : L_{a,\a}(f)= (1-\abs{a}^2)^\a(f\circ{\phi_a}-f(a)),\, a\in\D},
$$
on $D^p_s$, then in the context of $o-O$ types spaces introduced by Perfekt in \cite{PerfektDual} we have
$$
\MaD{\a}{p}{s} = M(D^p_s,\mathcal{L}_\a)\quad \text{and} \quad \maD{\a}{p}{s}=M_0(D^p_s,\mathcal{L}_\a).
$$
\begin{proposition}\label{propo: Tg compact mad}
	Let $p, s, \alpha $ in the admissible range. Let $g\in\U$, the following are equivalent:
	\begin{enumerate}
		\item[(i)] $T_g$ is compact on $\MaD{\a}{p}{s}$;
		\item[(ii)] $g\in \maD{0}{p}{s}$;
		\item[(iii)] $T_g$ is weakly compact on $\MaD{\a}{p}{s}$;
		\item[(iv)] $T_g(\MaD{\a}{p}{s})\subset \maD{\a}{p}{s}$.
	\end{enumerate}
\end{proposition}
\begin{proof}
	(i) implies (iii) by definition. \\
	To prove the equivalence between (iii) and (iv), first by \cite[Theorem 2.2]{PerfektDual} and the previous observation 
	$$
	(\maD{\a}{p}{s})^{\star \star} \simeq  \MaD{\a}{p}{s}.
	$$ 
	Also it can be checked that $M_\alpha(D^p_s) $ embeds continuously in $\mathcal{H}(\mathbb{D}) $ with the weak-$\star $ topology as a dual of $m_\alpha(D^p_s) $, therefore by an approximation argument  $T_g^{\star \star} = T_g$. 
	Then the equivalence follows from Gantmacher's theorem \cite[Theorem 5.23]{PositiveOperators}
	(iii) implies (ii), let $g\in U$, consider $\{a_n\}\subset \D$ such that $\abs{a_n}\to1^-$, without loss of generality we can assume $\lim_{n\to \infty}a_n =\zeta\in\T$, to prove that $g\in \maD{0}{p}{s}$ it is enough to check that
	$$
	\lim_{n\to \infty} \frac{1}{(1-\abs{a_n}^2)^{s-(p-2)}}\int_{S(a_n)}\abs{g'(z)}^p(1-\abs{z}^2)^sdA(z)=0.
	$$
	let 
	$$
	f_n(z)=\frac{1}{(1-\overline{a_n}z)^\a},\quad \text{and} \quad f_0(z) =\frac{1}{(1-\overline{\zeta}z)^\a},
	$$
	by Lemma~\ref{lemma: test funct a>0} we know that $sup_{n\in\N\cup\{0\}}\nm{f_n}_{\MaD{\a}{p}{s}}<\infty$. Consider now $h_n=f_n -f_0$, then these are also uniformly bounded on $\MaD{\a}{p}{s}$, in addition $f_n$ converges to $f$ uniformly on a bigger disk, then $\lim_{n\to \infty}\nm{h_n}_{D^p_s}=0$, then
	\begin{equation*}
		\begin{split}
			&\frac{1}{(1-\abs{a_n}^2)^{s-(p-2)}}\int_{S(a_n)}\abs{g'(z)}^p(1-\abs{z}^2)^sdA(z)\\
			&\asymp \frac{1}{(1-\abs{a_n}^2)^{s-(p(\a+1)-2)}}\int_{S(a_n)}\abs{f_n(z)g'(z)}^p(1-\abs{z}^2)^sdA(z) \\
			&\lesssim \frac{1}{(1-\abs{a_n}^2)^{s-(p(\a+1)-2)}}\int_{S(a_n)}\abs{(T_gf_0)'(z)}^p(1-\abs{z}^2)^sdA(z) + \nm{T_g(h_n)}_{\MaD{\a}{p}{s}}^p
		\end{split}
	\end{equation*}
	For the first term, by (iv) $T_g(\MaD{\a}{p}{s})\subset \maD{\a}{p}{s}$, in particular $T_gf_0 \in\maD{\a}{p}{s}$ implying
	$$
	\lim_{n\to \infty} \frac{1}{(1-\abs{a_n}^2)^{s-(p(\a+1)-2)}}\int_{S(a_n)}\abs{(T_gf_0)'(z)}^p(1-\abs{z}^2)^sdA(z) =0,
	$$
	for the second term, by \cite[Corollary 3.3]{PerfektCompact}, for each $\varepsilon>0$ there exists $N>0$ such that
	$$
	\nm{T_g(h_n)}_{\MaD{\a}{p}{s}} \leq N\nm{h_n}_{D^p_s} + \e \nm{h_n}_{\MaD{\a}{p}{s}},
	$$
	by applying that the family $h_n$ is uniformly bounded on $\MaD{\a}{p}{s}$ and the fact that $\lim_{n\to\infty}\nm{h_n}_{D^p_s}=0$, we conclude that
	$$
	\nm{T_g(h_n)}_{\MaD{\a}{p}{s}}=0.
	$$
	It remains to check that (ii) implies (i). First, as $g\in \maD{0}{p}{s}\cap\U$, the same argument of  Proposition~\ref{propo: Tg bounded Mad univalent} (i) shows that $g\in\maD{0}{p}{s-\beta}$, so there exists a sequence of polynomials $\{P_n\}$ such that $\lim_{n\to\infty}\nm{g-P_n}_{\MaD{0}{p}{s-\beta}}=0$. If we consider now the associated Volterra operators $T_{P_n}$, these are compact by Lemma~\ref{lemma: Tg compact poli}, and by Proposition~\ref{propo: Tg bounded Mad suff} 
	$$
	\nm{T_g-T_{P_n}}_{\MaD{\a}{p}{s}}=\nm{T_{g-P_n}}_{\MaD{\a}{p}{s}}\lesssim \nm{g-P_n}_{\MaD{0}{p}{s-\beta}}\rightarrow 0,
	$$
	so $T_g$ can be approximated in the operator norm by a sequence of compact operators, then $T_g$ is compact.
\end{proof}
\subsection{A technical Lemma}
The aim of this section is to prove a result in the spirit of \cite[Proposition 3.5]{ChalDask} in the context of $M_\a(D^p_s)$. We deal with the cases $\a=0$ and $\a>0$ by separated ways. First, if $\a=0$ the Möbius invariant behaviour of $M_0(D^p_s)$ allows us to adapt \cite[Proposition 3.5]{ChalDask} without much effort. 
\begin{proposition}\label{lemma: Tg log}
	Let $1<p<\infty$, $p-2<s<\infty$ and $g\in\MaD{0}{p}{s}\setminus F_{\log,0}(p,p-2,s-(p-2))$. There exists a function $F\in\MaD{0}{p}{s}$ such that $T_gF\in\MaD{0}{p}{s}\setminus\maD{0}{p}{s}.$
\end{proposition}
\begin{proof}
	If $g\in\MaD{0}{p}{s}\setminus \maD{0}{p}{s}$, then the function $F\equiv 1$ satisfies the desired properties. On the other hand, if $g\in F_{\log}(p,p-2,s-(p-2))\setminus F_{\log,0}(p,p-2,s-(p-2))$, then by Theorem~\ref{thlet: tg log caract} we have that $T_g(\MaD{0}{p}{s})\subset \MaD{0}{p}{s}$ and $T_g(\MaD{0}{p}{s})\nsubseteq \maD{0}{p}{s}$; therefore we can find a function $F$ as claimed in the thesis of the theorem.\\
	Therefore we can reduce the problem to the case $g\in\maD{0}{p}{s}\setminus F_{\log}(p,p-2,s-(p-2))$. For a point $w$, denote by $w^\star$ the hyperbolic midpoint between $0$ and $w$ and define the function
	$$
	\beta_w(z) = \log\frac{e}{1-\phi_{w^\star}(z)\overline{w}},
	$$ 
	the essential properties of these functions are summarized in \cite[Lemma 3.4]{ChalDask}. Moreover, by the Möbius invariance of $M_0(D^p_s)$ and Lemma~\ref{lemma: test funct a>0} we obtain that $\nm{\beta_w}_{M_0(D^p_s)}\lesssim 1$ for every $w\in \D$. Then we can essentially follow the proof of \cite[Proposition 3.5]{ChalDask} (see also Proposition~\ref{propo: construction function general}) to construct a sequence of intervals $I_n$ and a sequence of functions $F_n$ of the form 
	$$
	F_n(z)=\sum_{k=0}^n a_k \beta_{w_k}(z),
	$$
	for some $w_k\in \D$ such that
	\begin{enumerate}
		\item $0\leq a_k\leq 2^{-k}$;
		\item for all $n\in \N$ it holds
		$$
		\frac{1}{\abs{I_n}^{s-(p-2)}}\int_{S(I_n)}(\Re F_n(z))^p\abs{g'(z)}^p (1-\abs{z}^2)^sdA(z) \geq 1;
		$$
		\item for all $n\in \N$ we have $\nm{T_g F_n}_{M_0(D^p_s)}\leq \max\{\nm{T_g F_{n-1}}_{M_0(D^p_s)}+2^{-n}C(g),C(g)\}$, where $C(g)$ is a positive constant which depends only on $g$.
	\end{enumerate}
	Once the sequences are constructed, the function $F:= \sum_{k=0}^\infty a_k\beta_{w_k}$ verifies the desired hypothesis. $F\in M_0(D^p_s)$ as
	$$
	\nm{F}_{M_0(D^p_s)} \lesssim \sum_{k=0}^\infty a_k \nm{\beta_{w_k}}_{M_0(D^p_s)}\lesssim \sum_{k=0}^\infty a_k<\infty.
	$$
	Now, applying repeatedly property (3), we find that
	$$
	\nm{T_g F_n}_{M_0(D^p_s)}\leq \max\{\nm{T_g F_{0}}_{M_0(D^p_s)}+\sum_{r=1}^n2^{-r}C(g),C(g)\} \leq \nm{g}_{M_0(D^p_s)}+C(g),
	$$
	then by Fatou's lemma, if $I\subset \T$
	\begin{equation*}
		\begin{split}
			\frac{1}{\abs{I}^{s-(p-2)}}\int_{S(I)} \abs{g'(z)F(z)}^p(1-\abs{z}^2)^sdA(z) &\leq \liminf_{n\to \infty} \frac{1}{\abs{I}^{s-(p-2)}}\int_{S(I)} \abs{g'(z)F_n(z)}^p(1-\abs{z}^2)^sdA(z) \\
			&\leq \liminf_{n\to \infty} \nm{T_gF_n}_{\MaD{0}{p}{s}}^p \leq \nm{g}_{\MaD{0}{p}{s}}+C(g),
		\end{split}
	\end{equation*}
	so $T_gF\in\MaD{0}{p}{s}$. Finally, $T_g F\notin m_0(D^p_s)$ follows easily from (2)
	\begin{equation*}
		\begin{split}
			&\frac{1}{\abs{I_n}^{s-(p-2)}}\int_{S(I_n)} \abs{g'(z)F(z)}^p(1-\abs{z}^2)^sdA(z) \\
			&\geq \frac{1}{\abs{I_n}^{s-(p-2)}}\int_{S(I_n)}(\Re F_n(z))^p\abs{g'(z)}^p (1-\abs{z}^2)^sdA(z) \geq 1.
		\end{split}
	\end{equation*}
\end{proof}
Now dealing with the case $\a>0$ presents two main difficulties. Firstly, the lack of Möbius invariance force us to obtain a different type of test functions, secondly we will need to require stronger conditions in the hypothesis for our symbol $g$. In this first lemma, we construct some test functions.
\begin{lemma}\label{lemma: test function log}
Let $1<p<\infty$, $p-2<s<\infty$ and $0< \a <\frac{s-(p-2)}{p}$. If $g\in \bigcap_{0<\beta<\frac{s-(p-2)}{p}}\MaD{\beta}{p}{s}$, then there exists a family of functions $\{\beta_w^\a\}_{w\in\D}$ such that.
\begin{enumerate}
	\item[(i)] $\sup_{w\in\D}\nm{\beta_w^\a}_{\MaD{\a}{p}{s}}< \infty$;
	\item[(ii)] $\beta_w^\a\in H^\infty$ for each $w\in\D$;
	\item[(iii)] For every $z\in S(w)$ 
	$$
	\abs{\beta_w^\a(z)}\gtrsim \frac{1}{(1-\abs{w}^2)^\a};
	$$
	\item[(iv)] For all $\delta>0$ there exists $\delta'>0$ such that,
	$$
	\sup_{\abs{I}\geq \delta}\frac{1}{\abs{I}^{s-(p(\a+1)-2)}}\int_{S(I)} \abs{g'(z)\beta_w^\a(z)}^p(1-\abs{z}^2)^s dA(z) \leq 1,\quad 1-\abs{w}\leq\delta'.
	$$
\end{enumerate}
\end{lemma}

\begin{proof}
	If $\a>0$, for every $\lambda>0$ and $w\in \D$ consider the test functions presented in Lemma~\ref{lemma: test funct a>0},
	$$
	f_{w,\alpha,\lambda}(z) = \frac{(1-\abs{a}^2)^\lambda}{(1-\overline{w}z)^{\a + \lambda}},
	$$
	for this family of functions (i), (ii) and (iii) are straightforward, to obtain (iv) we fix $\lambda = \frac{s-(p-2)}{p}-\a$ and take $\beta_w^\a=f_{w,\alpha,\lambda}$. Lets check that this family verifies (iv), first simply notice that
	$$
	\sup_{\abs{I}\geq \delta}\frac{\int_{S(I)} \abs{g'(z)\beta_w^\a(z)}^p(1-\abs{z}^2)^s dA(z)}{\abs{I}^{s-(p(\a+1)-2)}}\leq \frac{(1-\abs{w}^2)^{s-(p(\a+1)-2)}}{\delta^{s-(p(\a+1)-2)}}\int_\D \abs{g'(z)}^p\frac{(1-\abs{z})^s}{\abs{1-\overline{w}z}^{s-(p-2)}}dA(z),
	$$
	lets focus in the integral in the right-hand side.
	
	Following the idea in the proof of Proposition~\ref{propo: Tg bounded Mad suff}, fix $w\in\D\setminus\{0\}$ and consider $N(w)=\min\set{n\in\N: 2^n(1-\abs{w}^2)\geq 1}$, for every $0\leq k < N(w)$ set $S_k(w)=S\left ((1-2^k(1-\abs{w}^2))e^{i\arg{w}} \right )$, $S_{-1}(w)=\emptyset$ and $S_{N(w)}=\D$. Then for every $0<\gamma<\frac{s-(p-2)}{p}$
	\begin{equation*}
		\begin{split}
			\int_\D \abs{g'(z)}^p\frac{(1-\abs{z})^s}{\abs{1-\overline{w}z}^{s-(p-2)}}dA(z) &= \sum_{k=0}^{N(w)} \int_{S_k(w)\setminus S_{k-1}(w)}\abs{g'(z)}^p\frac{(1-\abs{z}^2)^s}{\abs{1-\overline{w}z}^{s-(p-2)}}dA(z) \\
			&\lesssim (1-\abs{w}^2)^{-(s-(p-2))}\sum_{k=0}^{N(w)}\frac{1}{2^{k(s-(p-2))}}\int_{S_k(w)}\abs{g'(z)}^p(1-\abs{z}^2)^sdA(z) \\
			&\leq (1-\abs{w}^2)^{-p\gamma}\nm{g}_{\MaD{\gamma}{p}{s}}^p\sum_{k=0}^{N(w)}\frac{1}{2^{kp\gamma}}\\
			&\leq C(\gamma) (1-\abs{w}^2)^{-p\gamma}\nm{g}_{\MaD{\gamma}{p}{s}}^p.
		\end{split}
	\end{equation*} 
	Fixing $\gamma =\frac{s-(p-2)}{2p}-\frac{\a}{2}$, we obtain
	$$
	\sup_{\abs{I}\geq \delta}\frac{\int_{S(I)} \abs{g'(z)\beta_w^\a(z)}^p(1-\abs{z}^2)^s dA(z)}{\abs{I}^{s-(p(\a+1)-2)}}\lesssim \frac{(1-\abs{w}^2)^{\frac{s-(p-2)-p\a}{2}}}{\delta^{s-(p(\a+1)-2)}}\nm{g}_{\MaD{\gamma}{p}{s}}^p,
	$$
	by hypothesis $g\in M_\gamma(D^p_s)$ and $\frac{s-(p-2)-p\a}{2}>0$, so fixed $\delta>0$ there exists $\delta'>0$ such that
	$$
	\sup_{\abs{I}\geq \delta}\frac{\int_{S(I)} \abs{g'(z)\beta_w^\a(z)}^p(1-\abs{z}^2)^s dA(z)}{\abs{I}^{s-(p(\a+1)-2)}}\leq 1,\quad 1-\abs{w}\leq \delta'.
	$$
\end{proof}
\begin{proposition}\label{propo: construction function general}
	Let $1<p<\infty$, $p-2<s<\infty$ and $0<\a <\frac{s-(p-2)}{p}$. If $g\in \bigcap_{0<\beta<\frac{s-(p-2)}{p}}\MaD{\beta}{p}{s}\setminus \MaD{0}{p}{s}$.
	Then there exists a function $F\in\MaD{\a}{p}{s}$ such that $T_gF\in\MaD{\a}{p}{s}\setminus \maD{\a}{p}{s}$.
\end{proposition}
\begin{proof}
	We are going to construct a sequence of arcs ${I_n}$, a sequence of positive numbers $\{\delta_n\}$ and a sequence of functions $F_n$ of the form
	$$
	F_n(z)=\sum_{k=0}^n a_k\beta_{w_k}^\a(z),
	$$
	for some $w_k\in\D$, such that
	\begin{enumerate}
		\item $0\leq a_k \leq 2^{-k}$,
		\item $\abs{I_k}>\delta_n$ if $k\leq n$,
		\item for all $n\in\N$ it holds
		$$
		\frac{1}{\abs{I_n}^{s-(p(\a+1)-2)}}\int_{S(I_n)} \abs{F_n(z)g'(z)}^p(1-\abs{z}^2)^sdA(z)\geq \frac{1}{2},
		$$
		\item for all $n\in \N$ we have $\nm{T_gF_n}_{\MaD{\a}{p}{s}}\leq \max \left ( \nm{T_g F_{n-1}}_{\MaD{\a}{p}{s}}+2^{-n}C(g),C(g)\right )$, where $C(g)$ is a positive constant only depending on $g$.
	\end{enumerate}
	We will do a recursive definition. First, take $a_0=1$, $w_0=1$, $I_0=\T$ and $\delta_0=1$, so $F_0\equiv 1$. For the recursive step, suppose $I_1,\ldots I_{n-1}$, $w_1,\ldots,w_{n-1}$, $\delta_1,\ldots\delta_{n-1}$ and $a_1,\ldots,a_{n-1}$.
	Since $g\in \bigcap_{0<\beta<\frac{s-(p-2)}{p}}\MaD{\beta}{p}{s}\subset\maD{\a}{p}{s}$ and $F_{n-1}\in H^\infty$ by Lemma~\ref{lemma: test function log} (ii), we can find $0<\delta_n<\abs{I_{n-1}}$ such that
	\begin{equation}\label{eq: propo tg log eq1}
		\sup_{\abs{I}\leq \delta_n}\frac{1}{\abs{I_n}^{s-(p(\a+1)-2)}}\int_{S(I)}\abs{F_{n-1}(z)g'(z)}^p(1-\abs{z}^2)^sdA(z)\leq \frac{1}{2^p},
	\end{equation}
	now by Lemma~\ref{lemma: test function log}, there exists $0<\delta_n'<\delta_n$ such that if $1-\abs{w}<\delta_n'$ then
	\begin{equation}\label{eq: propo tg log eq2}
		\sup_{\abs{I}\geq \delta_n}\frac{1}{\abs{I_n}^{s-(p(\a+1)-2)}}\int_{S(I)}\abs{\beta^\a_w(z)g'(z)}^p(1-\abs{z}^2)^sdA(z)\leq 1.
	\end{equation}
	Furthermore, by Lemma~\ref{lemma: test function log} (iii), since $g\notin \MaD{0}{p}{s}$, there exists some $w_n\in\D$ with $1-\abs{w_n}\leq\delta_n'$ such that
	$$
	\frac{1}{\abs{I_{w_n}}^{s-(p(\a+1)-2)}}\int_{S(I_{w_n})}\abs{\beta^\a_{w_n}(z)g'(z)}^p(1-\abs{z}^2)^sdA(z)\geq 2^{pn}.
	$$
	On the other hand, as $g\in\MaD{\a}{p}{s}$, the supremum
	\begin{equation}\label{eq: propo tg log eq3}
	M_n^p:=\sup_{\abs{I}\leq \delta_n}\frac{1}{\abs{I}^{s-(p-2)}}\int_{S(I)}\abs{\beta^\a_{w_n}(z)g'(z)}^p(1-\abs{z}^2)^sdA(z) \geq 2^{pn},
	\end{equation}
	is in fact a maximum which is obtained for some interval $I_n\subset \T$, $\abs{I_n}\leq \delta_n'$. We claim that the function
	$$
	F_n:= F_{n-1} + M_n^{-1}\beta_{w_n}
	$$
	satisfies the required properties. \\
	Clearly by \eqref{eq: propo tg log eq3} $a_n=M_n^{-1}\leq 2^{-n}$, and (2) holds by definition of $\delta_n$. \\
	(3) follows from \eqref{eq: propo tg log eq1} and \eqref{eq: propo tg log eq3}
	\begin{equation*}
		\begin{split}
			&\left (\frac{1}{\abs{I_n}^{s-(p(\a+1)-2)}}\int_{S(I_n)} \abs{F_n(z)g'(z)}^p(1-\abs{z}^2)^sdA(z)\right )^{\frac{1}{p}} 
			\\
			&\geq \left (\frac{M_n^{-p}}{\abs{I_n}^{s-(p(\a+1)-2)}}\int_{S(I_n)} \abs{\beta_{w_n}^\a(z)g'(z)}^p(1-\abs{z}^2)^sdA(z)\right )^{\frac{1}{p}} \\
			&- \left (\frac{1}{\abs{I_n}^{s-(p(\a+1)-2)}}\int_{S(I_n)} \abs{F_{n-1}(z)g'(z)}^p(1-\abs{z}^2)^sdA(z)\right )^{\frac{1}{p}} \geq 1-\frac{1}{2}=\frac{1}{2}.
		\end{split}
	\end{equation*}
	Finally, to obtain (4), let $I\subset \T$, we consider two different cases.
	First, if $\abs{I}\geq \delta_n$, then by \eqref{eq: propo tg log eq2}
	\begin{equation*}
		\begin{split}
			&\left (\frac{1}{\abs{I}^{s-(p(\a+1)-2)}}\int_{S(I)}\abs{g'(z)F_n(z)}^p(1-\abs{z}^2)^s \right)^{\frac{1}{p}} \\
			&\leq M_n^{-1}\left (\frac{1}{\abs{I}^{s-(p(\a+1)-2)}}\int_{S(I)}\abs{g'(z)\beta^\a_{w_n}(z)}^p(1-\abs{z}^2)^s \right)^{\frac{1}{p}} + \nm{T_gF_{n-1}}_{\MaD{\a}{p}{s}} \\
			&\leq 2^{-n}+\nm{T_gF_{n-1}}_{\MaD{\a}{p}{s}},
		\end{split}
	\end{equation*}
	on the other hand, if $\abs{I}<\delta_n$, then by \eqref{eq: propo tg log eq1} and \eqref{eq: propo tg log eq3}
	\begin{equation*}
		\begin{split}
			&\left (\frac{1}{\abs{I}^{s-(p(\a+1)-2)}}\int_{S(I)} \abs{F_n(z)g'(z)}^p(1-\abs{z}^2)^sdA(z)\right )^{\frac{1}{p}} 
			\\ &\leq \left (\frac{M_n^{-p}}{\abs{I}^{s-(p(\a+1)-2)}}\int_{S(I)} \abs{\beta^\a_{w_n}(z)g'(z)}^p(1-\abs{z}^2)^sdA(z)\right )^{\frac{1}{p}} \\
			&+ \left (\frac{1}{\abs{I}^{s-(p(\a+1)-2)}}\int_{S(I)} \abs{F_{n-1}(z)g'(z)}^p(1-\abs{z}^2)^sdA(z)\right )^{\frac{1}{p}} \leq 1 + \frac{1}{2} =C(g).
		\end{split}
	\end{equation*}

	Now that we have constructed the sequence, the proof can be finished as follows. First, by Lemma~\ref{lemma: test function log} (i)
	$$
	\sum_{k=0}^\infty a_k \nm{\beta_{w_k}^\a}_{\MaD{\a}{p}{s}} \lesssim \sum_{k=0}^\infty a_k < \infty,
	$$
	then $F:= \sum_{k=0}^{\infty} a_k\beta_{w_k}^\a\in \MaD{\a}{p}{s}$, lets check that $F$ satisfies the desired properties. \\
	By iterating (4), $$\nm{T_gF_n}_{\MaD{\a}{p}{s}}\leq \max \left ( \nm{T_g F_{0}}_{\MaD{\a}{p}{s}}+C(g)\sum_{r=0}^n2^{-r},C(g)\right )\leq \nm{g}_{\MaD{\a}{p}{s}}+C(g).$$
	Now, by Fatou's lemma, if $I\subset \T$
	\begin{equation*}
		\begin{split}
			\frac{1}{\abs{I}^{s-(p(\a+1)-2)}}\int_{S(I)} \abs{g'(z)F(z)}^p(1-\abs{z}^2)^sdA(z) &\leq \liminf_{n\to \infty} \frac{1}{\abs{I}^{s-(p(\a+1)-2)}}\int_{S(I)} \abs{g'(z)F_n(z)}^p(1-\abs{z}^2)^sdA(z) \\
			&\leq \liminf_{n\to \infty}\nm{T_gF_n}_{\MaD{\a}{p}{s}}^p \leq \nm{g}_{\MaD{\a}{p}{s}}+C(g),
		\end{split}
	\end{equation*}
	so $T_gF\in\MaD{\a}{p}{s}$. To check that $T_gF\notin\maD{\a}{p}{s}$, just notice that if $n$ is big, then by (2), (4) and \eqref{eq: propo tg log eq2}
	\begin{equation*}
		\begin{split}
			&\frac{1}{\abs{I_n}^{s-(p(\a+1)-2)}}\int_{S(I_n)}\abs{F(z)g'(z)}^p(1-\abs{z}^2)^sdA(z) \geq \frac{1}{\abs{I_n}^{s-(p(\a+1)-2)}}\int_{S(I_n)}\abs{F_n(z)g'(z)}^p(1-\abs{z}^2)^sdA(z) \\
			&- \sum_{k=n+1}^\infty a_k^p \frac{1}{\abs{I_n}^{s-(p(\a+1)-2)}}\int_{S(I_n)}\abs{\beta_{w_n}^\a(z)g'(z)}^p(1-\abs{z}^2)^sdA(z) \geq 1-\sum_{k=n+1}^\infty a_k^p >0.
		\end{split}
	\end{equation*}
\end{proof}
Lastly, an application of Proposition~\ref{propo: Tg compact mad} let us push a little further this last Proposition for univalent symbols.
\begin{corollary}\label{coro: construction mad}
	Let $p,s,\a>0$ in the admissible range. If $g\in \U$ such that $g\in \bigcap_{0<\beta<\frac{s-(p-2)}{p}}\MaD{\beta}{p}{s}\setminus \maD{0}{p}{s}$.
	Then there exists a function $F\in\MaD{\a}{p}{s}$ such that $T_gF\in\MaD{\a}{p}{s}\setminus \maD{\a}{p}{s}$.
\end{corollary}
\begin{proof}
	We distinguish two cases, on the one hand, if $g\in \bigcap_{0<\beta<\frac{s-(p-2)}{p}}\MaD{\beta}{p}{s}\setminus \MaD{0}{p}{s}$ the result follows from Proposition~\ref{propo: construction function general}. On the other hand, if $g\in \MaD{0}{p}{s}\setminus\maD{0}{p}{s}$ by Proposition~\ref{propo: Tg bounded Mad univalent} (ii), $T_{g}(\MaD{\a}{p}{s})\subset \MaD{\a}{p}{s}$ and by Proposition~\ref{propo: Tg compact mad} $T_{g}(\MaD{\a}{p}{s})\nsubseteq \maD{\a}{p}{s}$, so there exists $F\in \MaD{\a}{p}{s}$ such that $T_gF\in\MaD{\a}{p}{s}\setminus \maD{\a}{p}{s}$.
\end{proof}
\section{Main results on semigroups}
\subsection{First properties of $[\phi_t,M_\a(D^p_s)]$}
Our main objective is to study the maximal subspace of strong continuity for $\MaD{\a}{p}{s}$ in the admissible range of parameters. \\
\begin{lemma}\label{lemma: contains good space}
	Let $p,s,\a$ in the admissible range. Then the following holds:
	\begin{enumerate}
		\item[(i)] If $s>p-1$ then $H^\infty\subset \MaD{\a}{p}{s}$;
		\item[(ii)] If $2\leq p <\infty$ and $p-2\leq s \leq p-1$, then $B^p\subset \MaD{\a}{p}{s}$;
		\item[(iii)] If $1< p <2$ and $0\leq s \leq p-1$, then $\mathcal{D}\subset \MaD{\a}{p}{s}$.
	\end{enumerate}
\end{lemma}
\begin{proof}
	(i) Let $f\in H^\infty$ and $a\in\D$, by applying that $\abs{f'(z)}\lesssim \frac{\nm{f}_{H^\infty}}{1-\abs{z}^2}$, $z\in\D$ then
	\begin{equation*}
		\begin{split}
			\nm{f\circ\phi_a - f(a)}_{D^p_s}^p&\lesssim \nm{f}_{H^\infty}^p\int_{\D} \left (\frac{\abs{\phi_a'(z)}}{1-\abs{\phi_a(z)}^2}\right )^p (1-\abs{z}^2)^s dA(z)\\
			&= \nm{f}_{H^\infty}^p\int_\D (1-\abs{z}^2)^{s-p}dA(z)\lesssim \nm{f}_{H^\infty}^p,
		\end{split}
	\end{equation*}
	so $f\in\MaD{0}{p}{s}\subset\MaD{\a}{p}{s}$.\\
	(ii) By applying that Besov spaces are Mobiüs invariant and $s\geq p-2$, if $f\in B^p$ and $a\in\D$ then
	$$
	\nm{f\circ\phi_a - f(a)}_{D^p_s}^p \leq \nm{f\circ\phi_a - f(a)}_{B^p}^p = \nm{f-f(0)}_{B^p}^p \leq \nm{f}_{B^p}^p,
	$$
	so $f\in\MaD{0}{p}{s}\subset\MaD{\a}{p}{s}$. \\
	(iii) Let $f\in\mathcal{D}$, in this case, we are going to use that, by Proposition~\ref{prop: classif Mad} is it enough to prove that 
	$$
	\sup_{a\in\D}\int_\D \abs{f'(z)}^p\frac{(1-\abs{z}^2)^{s}(1-\abs{a}^2)^{s-(p-2)}}{\abs{1-\overline{a}z}^{2(s-(p-2))}}dA(z)<\infty.
	$$
	Let $a\in\D$ by Hölder's inequality and \cite[Lemma 3.10]{zhu}
	\begin{equation*}
		\begin{split}
			&\int_\D \abs{f'(z)}^p\frac{(1-\abs{z}^2)^{s}(1-\abs{a}^2)^{s-(p-2)}}{\abs{1-\overline{a}z}^{2(s-(p-2))}}dA(z)\\
			&\leq (1-\abs{a}^2)^{s-(p-2)} \left (\int_\D \abs{f'(z)}^2 dA(z) \right )^{\frac{p}{2}}\left (\int_\D \frac{(1-\abs{z}^2)^{\frac{2s}{2-p}}}{\abs{1-\overline{a}z}^{4+\frac{4s}{2-p}}}dA(z) \right)^{\frac{2-p}{2}} \\
			&\asymp  \nm{f}_{\mathcal{D}}^p (1-\abs{a}^2)^{s-(p-2)}\left (\frac{1}{(1-\abs{a}^2)^{2+\frac{2s}{2-p}}} \right )^{\frac{2-p}{2}}= \nm{f}_{\mathcal{D}}^p,
		\end{split}
	\end{equation*}
	so $f\in\MaD{0}{p}{s}\subset\MaD{\a}{p}{s}$.
\end{proof}
\begin{proposition}\label{propo: small contained in cont}
	Let $\a,p,s$ in the admissible range and $(\phi_t)$ a semigroup on $\D$, then:
	\begin{enumerate}
		\item[(i)] If $\a\geq \frac{s-(p-2)}{p}$, then 
		$$
		[\phi_t,D^p_s]=[\phi_t,\MaD{\a}{p}{s}] = \MaD{\a}{p}{s} =D^p_s;
		$$
		\item[(ii)] If $s>p-2$ and $0\leq \a < \frac{s-(p-2)}{p}$, then
		$$
		\maD{\a}{p}{s}\subset [\phi_t,\MaD{\a}{p}{s}].
		$$
	\end{enumerate}
\end{proposition}
\begin{proof}
	Let $X_0$ be the closure of polynomials in $\MaD{\a}{p}{s}$, if we prove that $X_0\subset [\phi_t,\MaD{\a}{p}{s}]$ the result follows. \\
	Let $f\in X_0$ and $\{P_n\}_{n\in\N}$ a sequence of polynomials such that $\lim_{n\to \infty}\nm{f-P_n}_{\MaD{\a}{p}{s}}=0$, then for every $t\geq 0$
	\begin{equation*}
		\begin{split}
			\nm{f\circ\phi_t-f}_{\MaD{\a}{p}{s}} &\leq \nm{(f-P_n)\circ\phi_t}_{\MaD{\a}{p}{s}} + \nm{P_n\circ\phi_t-P_n}_{\MaD{\a}{p}{s}} + \nm{P_n-f}_{\MaD{\a}{p}{s}} \\
			&\leq (1+\nm{C_t}_{\MaD{\a}{p}{s}}) \nm{f-P_n}_{\MaD{\a}{p}{s}} + \nm{P_n\circ\phi_t-P_n}_{\MaD{\a}{p}{s}}.
		\end{split}
	\end{equation*}
	By the bounds for the norm of composition operators obtained in Theorem~\ref{thm:1.1}, 
	$$
	\sup_{t\in[0,1]}\nm{C_t}_{\MaD{\a}{p}{s}} <\infty,
	$$
	so by applying that $\lim_{n\to \infty}\nm{f-P_n}_{\MaD{\a}{p}{s}}=0$, the first term goes to zero, so to end the proof we just need to check that for every polynomial $P$
	$$
	\lim_{t\to 0^+}\nm{P\circ \phi_t-P}_{\MaD{\a}{p}{s}}=0.
	$$
	In fact, by Lemma~\ref{lemma: contains good space} if we check that for every polynomial
	$$
	\lim_{t\to 0^+}\nm{P\circ \phi_t-P}_{X}=0
	$$
	for $X=H^\infty$, $B^p$ and $\mathcal{D}$, we can conclude the proof, and this results are already well-known. For $H^\infty$ by \cite[Proposition 3.2]{Gumenyuk}, for $B^p$ with $p\geq 2$ by \cite[Corollary 3.3]{AndJovSmithBp} and for $\mathcal{D}$ by \cite[Theorem 1]{SiskakisDirichlet}, so the result holds.
\end{proof}
In particular, we obtained the following generalization of \cite{SiskakisDirichlet}.
\begin{corollary}
Let $p, s$ in the admissible range, then every semigroup $(\phi_t)$ on $\D$ induces an strongly continuous semigroup of composition operators on the weighted Dirichlet space $D^p_s$.
\end{corollary}
The last part of this section is devoted to checking that, whenever $\MaD{\a}{p}{s}$ is non-separable then  $[\phi_t,\MaD{\a}{p}{s}]\subsetneq\MaD{\a}{p}{s}$ for every semigroup $(\phi_t)$ on $\D$. In this case we will work by separated ways the case $\a=0$ and $\a>0$. \\
The following lemma follows from \cite[Remark 1]{DaskGal}(see also \cite[Theorem 1.1]{AndJovSmithHinf}).
\begin{letterlemma}\label{lemmalet: general anderson}
	Let $X\subset \H(\D)$ be a Banach space and $(\phi_t)$ a semigroup on $\D$. Suppose that  $\bigcap\limits_{s\in(0,\infty)}Q_s\subset X \subset \B$, then $[\phi_t,X]\subsetneq X$.
\end{letterlemma}
\begin{proposition}\label{prop: proper subs 1}
Let $p, s $ in the admissible range and $(\phi_t)$ a semigroup on $\D$, then
	$$
	[\phi_t,\MaD{0}{p}{s}]\subsetneq\MaD{0}{p}{s}.
	$$
\end{proposition}
\begin{proof}
	Joining Proposition~\ref{prop: classif Mad}, \cite[Corollary 3.6]{zhaosurvey} and Lemma~\ref{lemma: contains Qp}, we obtain that
	$$
	\bigcap\limits_{s\in(0,\infty)}Q_s \subset \MaD{0}{p}{s} \subset \B,
	$$
	and the result follows from Lemma~\ref{lemmalet: general anderson}.
\end{proof}
For the case $\a>0$, we take a different approach. In the context of Morrey spaces (remember $H^{2,\lambda} = \MaD{\frac{1-\lambda}{2}}{2}{1}$), Galanopoulos, Merchan and Siskakis \cite{GalMerchanSiskakis}, proved the following.
\begin{lettertheorem}\cite[Theorem 3.5]{GalMerchanSiskakis} Let $0<\a<\frac{1}{2}$ and let $X$ be a Banach space of analytic functions on $\D$ such that $\MaD{\a}{2}{1}\subset X \subset \B^{\a+1}$. Then for every  semigroup $(\phi_t)$ on $\D$ we have $[\phi_t,X]\subsetneq X$. In particular, $[\phi_t,\MaD{\a}{2}{1}]\subsetneq \MaD{\a}{2}{1}$.
\end{lettertheorem}
The proof of this result is based on the following lemma.
\begin{letterlemma}
	Let $0<\a<\frac{1}{2}$, for every semigroup $(\phi_t)$ on $\D$, there is a function $f\in\MaD{\a}{2}{1}$ such that
	$$
	\liminf_{t\to 0}\nm{f\circ \phi_t - f}_{\B^{\a+1}} \geq 1.
	$$
\end{letterlemma}
In fact, the functions of this lemma are just test functions 
$$
f_{\zeta,\a}(z)=\frac{1}{(1-\overline{\zeta}z)^{\a}},\quad \zeta\in\T,\, z\in\D, \, 0<\a<\frac{1}{2} .
$$
with a particular choice of $\zeta\in\T$ depending only on the semigroup.\\
Then by Lemma~\ref{lemma: test funct a>0} we can mimick the proof of \cite[Theorem 3.5]{GalMerchanSiskakis} to obtain an equivalent result for our whole range of space.
\begin{proposition}\label{prop: proper subs 2}
	Let $p, s, \alpha>0 $ in the admissible range, and let $X$ be a Banach space of analytic functions on $\D$ such that $\MaD{\a}{p}{s}\subset X \subset \B^{\a+1}$. Then for every  semigroup $(\phi_t)$ on $\D$ we have $[\phi_t,X]\subsetneq X$.
\end{proposition}

\begin{Prf}{\em{  Theorem~\ref{thm:1.2}}.}
	The embedding $m_\a(D^p_s)\subset[\phi_t,M_\a(D^p_s)]$ follows from Proposition~\ref{propo: small contained in cont}, on the other hand, $[\phi_t,M_\a(D^p_s)]$ is a proper subspace of $M_\a(D^p_s)$ by Proposition~\ref{prop: proper subs 1} in the case $\a=0$ and by Propostion~\ref{prop: proper subs 2} if $\a>0$.
\end{Prf}

\subsection{Characterize $[\phi_t,\MaD{\a}{p}{s}]=\maD{\a}{p}{s}$} 
Now we deal with the proof of Theorems\ref{thm:1.3} and \ref{thm:1.4}.

\begin{Prf}{\em{  Theorem~\ref{thm:1.3}}.}
	We start by proving the sufficiency of the condition, i.e, (ii) implies (i), let $(\phi_t)$ be an elliptic semigroup, without loss of generality we can assume that the Denjoy-Wolff's point is $\tau=0$, then 
	$$
	\gamma'(z)=\frac{z}{G(z)}=-\frac{1}{p(z)},
	$$
	since $\Re{p}\geq 0$, by Alexander-Noshiro-Warschawski criterion, $\gamma$ is univalent and by hypothesis $G$ satisfies the vanishing logaritmich Bloch condition, thats equivalent to $\gamma\in \B_{\log,0}$. 
	Then by Corollary~\ref{coro: Blog univ equiv Flog} and Theorem~\ref{thlet: tg log caract} (iv), $T_\gamma(M_0(D^p_s))\subset m_0(D^p_s)$, so the result follows from Theorem~\ref{thm:1.2} and \eqref{eq:Blasco}
	$$
	\maD{0}{p}{s}\subset[\phi_t,\MaD{0}{p}{s}]= \overline{(T_\gamma(\MaD{0}{p}{s})\oplus \C)\cap \MaD{0}{p}{s}}=\overline{T_\gamma(\MaD{0}{p}{s})\oplus \C} \subset \maD{0}{p}{s}.
	$$ 
	Lets prove the converse now, assume that $(\phi_t)$ and $[\phi_t,\MaD{0}{p}{s}]=\maD{0}{p}{s}$, we have two different cases, in terms of the classification of the semigroup. 
	
	If $(\phi_t)$ is an elliptic semigroup, arguing as before $\gamma$ is univalent and by Harnack's inequality belongs to $\B$, so by Lemma~\ref{lemma: univalent bloch equiv maD}, we have that $\gamma\in\MaD{0}{p}{s}$, are going to prove that $\gamma\in F_{\log,0}(p,p-2,s-(p-2))$. To prove this, assume that $\gamma\in\MaD{0}{p}{s}\setminus F_{\log,0}(p,p-2,s-(p-2)) $. By Proposition~\ref{lemma: Tg log}, there is a function $F\in\MaD{0}{p}{s}$ such that $T_\gamma F\in\MaD{0}{p}{s}\setminus\maD{0}{p}{s}$
	but by \eqref{eq:Blasco}
	$$
	T_\gamma F\in T_\gamma(\MaD{0}{p}{s})\cap \MaD{0}{p}{s}) \subset [\phi_t,\MaD{0}{p}{s}]= \maD{0}{p}{s},
	$$
	and this is a contradiction. So $\gamma\in F_{\log,0}(p,p-2,s-(p-2))$ and by Corollary~\ref{coro: Blog univ equiv Flog} thats equivalent to $\gamma\in \B_{\log,0}$, i.e, $G$ satisfies the vanishing logaritmich Bloch condition.\\
	
	It remains to prove that non-elliptic semigroups cannot satisfy $[\phi_t,\MaD{0}{p}{s}]=\maD{0}{p}{s}$, by contradiction assume that $(\phi_t)$ is non-elliptic and $[\phi_t,\MaD{0}{p}{s}]=\maD{0}{p}{s}$, without loss of generality assume that the Denjoy-Wolff's point is $\tau=1$. Let $h$ be the Koening function and $H(z)=\log (h(z)-w_0)$ for $w_0\in \C\setminus \overline{h(\D)}$, then by Koebe's distortion theorem $H\in \B$, furthermore, as $h(\D)$ is starlike at infinity and $h-w_0$ is an univalent non-zero function, $H$ is also univalent so by Lemma~\ref{lemma: univalent bloch equiv maD}  $H\in\MaD{0}{p}{s}$. We distinguish now two different cases.
	
	If $H\in \MaD{0}{p}{s}\setminus F_{\log,0}(p,p-2,s-(p-2))$ by Proposition~\ref{lemma: Tg log}, there is a function $F\in\MaD{0}{p}{s}$, such that $T_{H}F \in \MaD{0}{p}{s}\setminus\maD{0}{p}{s}$, notice that this is equivalent to $T_h(\frac{F}{h-w_0}) \in \MaD{0}{p}{s}\setminus\maD{0}{p}{s}$, since for this case $\gamma=h$, if we prove that $\frac{F}{h-w_0}\in \MaD{0}{p}{s}$, then by \eqref{eq:Blasco} we will obtain a contradiction. For every $z\in\D$
	$$
	\abs{\left (\frac{F}{h-w_0}\right )'(z)}^p\lesssim \frac{\abs{F'(z)}^p}{\abs{h(z)-w_0}^p} + \frac{\abs{F(z)h'(z)}^p}{\abs{h(z)-w_0}^{2p}} \leq \nm{(h-w_0)^{-1}}_{H^\infty}^p (\abs{F'(z)}^p+\abs{(T_{H}F)'(z)}^p),
	$$
	then if $I\subset \T$
	$$
	\frac{1}{\abs{I}^{s-(p-2)}}\int_{S(I)}\abs{\left (\frac{F}{H}\right )'(z)}^p(1-\abs{z}^2)^sdA(z) \lesssim  \nm{(h-w_0)^{-1}}_{H^\infty}^p \left (\nm{F}_{\MaD{0}{p}{s}}^p+\nm{T_{H}F}_{\MaD{0}{p}{s}}^p \right ),
	$$
	so $\frac{F}{h-w_0}\in \MaD{0}{p}{s}$ and $T_h(\frac{F}{H})\in [\phi_t,\MaD{0}{p}{s}]=\maD{0}{p}{s}$, leading to a contradiction. 
	
	On the other hand, if $H\in F_{\log,0}(p,p-2,s-(p-2))$, then by Theorem~\ref{thlet: tg log caract} we know that the operator $T_H$ is compact on $\MaD{0}{p}{s}$. Then arguing as in the proof of \cite[Theorem 1.1]{ChalDask}, $T_{H}$ has trivial spectrum. In particular, there exists $f\in \MaD{0}{p}{s}$ such that
	$$
	f(z)-T_{H}f(z)\equiv 1.
	$$
	
	Solving this first order ODE we find that $f=h-w_0$, in particular $h\in\MaD{0}{p}{s}$. In fact, as $h'(z)G(z)\equiv i\in\MaD{0}{p}{s}$ by \eqref{eq: maximal subs 1} it happens that $h\in[\phi_t,\MaD{0}{p}{s}]=\maD{0}{p}{s}$, so by Lemma~\ref{lemma: univalent bloch equiv maD} we get that $h\in \B_0$, but the range of the Koening's map always contains an infinite half-strip, implying that $h$ cannot belong to $\B_0$, leading to a contradiction and concluding the proof.
\end{Prf}

\begin{Prf}{\em{  Theorem~\ref{thm:1.4}}.}
	
	In this case, we start by proving the equivalence under the assumption of $(\phi_t)$ being an elliptic semigroup. Without loss of generality we can assume that the Denjoy-Wolff's point is $\tau=0$, then $\gamma$ is univalent and belongs to $\B_0$, in fact by Lemma~\ref{lemma: univalent bloch equiv maD} $\gamma\in\MaD{0}{p}{s}$, so $T_\gamma$ is always bounded on $\MaD{\a}{p}{s}$ by Proposition~\ref{propo: Tg bounded Mad univalent} (ii), then we can simplify \eqref{eq:Blasco} obtaining
	$$
	[\phi_t,\MaD{\a}{p}{s}]=\overline{T_\gamma(M_\alpha(D^p_s)) \oplus \mathbb{C}},
	$$
	so (i) is equivalent to the embedding $T_\gamma(M_\alpha(D^p_s))\subset m_\a(D^p_s)$, and by Proposition~\ref{propo: Tg compact mad} and Lemma~\ref{lemma: univalent bloch equiv maD}, this embedding is equivalent to $\gamma\in \B_0$, i.e, $G$ satisfying the vanishing Bloch condition. 
	
	It remains to prove that that non-elliptic semigroups cannot satisfy $[\phi_t,\MaD{\a}{p}{s}]=\maD{\a}{p}{s}$.\\
	Assume that $(\phi_t)$ is non-elliptic and $[\phi_t,\MaD{0}{p}{s}]=\maD{0}{p}{s}$, without loss of generality assume that the Denjoy-Wolff's point is $\tau=1$,  Let $h$ be the Koening function and $H(z)=\log (h(z)-w_0)$ for $w_0\in \C\setminus \overline{h(\D)}$, then $H$ is an univalent function that belongs to $\MaD{0}{p}{s}$, we distinguish two cases. 
	
	If $H\in\MaD{0}{p}{s}\setminus\maD{0}{p}{s}$, by Corollary~\ref{coro: construction mad} there exists a function $F\in\MaD{\a}{p}{s}$ such that $T_{H} F\in \MaD{\a}{p}{s}\setminus \maD{\a}{p}{s}$, this is equivalent to $T_h(\frac{F}{h-w_0})\in \MaD{\a}{p}{s}\setminus \maD{\a}{p}{s}$, now arguing as in the proof of Theorem~\ref{thm:1.3} we can check that $\frac{F}{h-w_0}\in \MaD{\a}{p}{s}$ and obtain a contradiction. 
	
	If $H\in\maD{0}{p}{s}$ by Proposition~\ref{propo: Tg compact mad} $T_{\log H}$ is compact on $\MaD{\beta}{p}{s}$ for every $0<\beta<\frac{s-(p-2)}{p}$, so again arguing as in the proof of Theorem~\ref{thm:1.3} for each $\beta$ there exists $f_\beta\in\MaD{\beta}{p}{s}$ such that
	$$
	f_{\beta}(z)-T_{\log H}f_\beta(z)\equiv 1,
	$$
	and solving the first order ODE, we find that $f_\beta=h-w_0$, so in particular $h\in\MaD{\beta}{p}{s}$ for every  $0<\beta<\frac{s-(p-2)}{p}$, it means 
	$$
	h\in \bigcap_{0<\beta<\frac{s-(p-2)}{p}} \MaD{\beta}{p}{s}\setminus \maD{0}{p}{s},
	$$
	Notice that $h\notin \maD{0}{p}{s}$ follows from Lemma~\ref{lemma: univalent bloch equiv maD} and the fact that $h\notin \B_0$. Lastly, by Corollary~\ref{coro: construction mad} there exists a function $F\in\MaD{\a}{p}{s}$, such that $T_hF\in \MaD{\a}{p}{s}\setminus \maD{\a}{p}{s}$ obtaining a contradiction and concluding to proof.
\end{Prf}

\bibliographystyle{plain}
\bibliography{biblio}

@article{Arevalo2019,
 author = {Ar{\'e}valo, Irina and Contreras, Manuel D. and Rodr{\'{\i}}guez-Piazza, Luis},
 title = {Semigroups of composition operators and integral operators on mixed norm spaces},
 fjournal = {Revista Matem{\'a}tica Complutense},
 journal = {Rev. Mat. Complut.},
 issn = {1139-1138},
 volume = {32},
 number = {3},
 pages = {767--798},
 year = {2019},
 language = {English},
 doi = {10.1007/s13163-019-00300-7},
 zbMATH = {7123184},
 Zbl = {1435.30150}
}

@book{Garnett2006,
 author = {Garnett, John B.},
 title = {Bounded analytic functions},
 edition = {Revised 1st ed.},
 fseries = {Graduate Texts in Mathematics},
 series = {Grad. Texts Math.},
 issn = {0072-5285},
 volume = {236},
 isbn = {0-387-33621-4},
 year = {2006},
 publisher = {New York, NY: Springer},
 language = {English},
 doi = {10.1007/0-387-49763-3},
 zbMATH = {5062981},
 Zbl = {1106.30001}
}

@book{PositiveOperators,
 author = {Aliprantis, Charalambos D. and Burkinshaw, Owen},
 title = {Positive operators},
 fseries = {Pure and Applied Mathematics (Academic Press)},
 series = {Pure Appl. Math., Academic Press},
 issn = {0079-8169},
 volume = {119},
 year = {1985},
 publisher = {Academic Press, New York, NY},
 language = {English},
 zbMATH = {3983937},
 Zbl = {0608.47039}
}

@article{AndJovSmithHinf,
 author = {Anderson, Austin and Jovovic, Mirjana and Smith, Wayne},
 title = {Composition semigroups on {BMOA} and {{\(H^\infty\)}}},
 fjournal = {Journal of Mathematical Analysis and Applications},
 journal = {J. Math. Anal. Appl.},
 issn = {0022-247X},
 volume = {449},
 number = {1},
 pages = {843--852},
 year = {2017},
 language = {English},
 doi = {10.1016/j.jmaa.2016.12.032},
 zbMATH = {6675622},
 Zbl = {1369.47025}
}

@article{AndJovSmithBp,
 author = {Anderson, Austin and Jovovic, Mirjana and Smith, Wayne},
 title = {Composition semigroups on the {Besov} spaces},
 fjournal = {Complex Analysis and Operator Theory},
 journal = {Complex Anal. Oper. Theory},
 issn = {1661-8254},
 volume = {19},
 number = {3},
 pages = {17},
 note = {Id/No 64},
 year = {2025},
 language = {English},
 doi = {10.1007/s11785-025-01686-7},
 zbMATH = {8039905}
}

@article{BlascoEtal,
 author = {Blasco, Oscar and Contreras, Manuel D. and D{\'{\i}}az-Madrigal, Santiago and Mart{\'{\i}}nez, Josep and Papadimitrakis, Michael and Siskakis, Aristomenis G.},
 title = {Semigroups of composition operators and integral operators in spaces of analytic functions},
 fjournal = {Annales Academiae Scientiarum Fennicae. Mathematica},
 journal = {Ann. Acad. Sci. Fenn., Math.},
 issn = {1239-629X},
 volume = {38},
 number = {1},
 pages = {67--89},
 year = {2013},
 language = {English},
 doi = {10.5186/aasfm.2013.3806},
 zbMATH = {6195908},
 Zbl = {1273.30046}
}

@article{BlasiPau,
 author = {Blasi, Daniel and Pau, Jordi},
 title = {A characterization of {Besov}-type spaces and applications to {Hankel}-type operators},
 fjournal = {Michigan Mathematical Journal},
 journal = {Mich. Math. J.},
 issn = {0026-2285},
 volume = {56},
 number = {2},
 pages = {401--417},
 year = {2008},
 language = {English},
 doi = {10.1307/mmj/1224783520},
 zbMATH = {5493336},
 Zbl = {1182.46015}
}

@book{BracContDiazBook,
 author = {Bracci, Filippo and Contreras, Manuel D. and D{\'{\i}}az-Madrigal, Santiago},
 title = {Continuous semigroups of holomorphic self-maps of the unit disc},
 fseries = {Springer Monographs in Mathematics},
 series = {Springer Monogr. Math.},
 issn = {1439-7382},
 isbn = {978-3-030-36781-7; 978-3-030-36784-8; 978-3-030-36782-4},
 year = {2020},
 publisher = {Cham: Springer},
 language = {English},
 doi = {10.1007/978-3-030-36782-4},
 zbMATH = {7170296},
 Zbl = {1441.30001}
}

@article{ChalDask,
  title = {Holomorphic semigroups and Sarason’s characterization of vanishing mean oscillation},
  volume = {39},
  ISSN = {2235-0616},
  url = {http://dx.doi.org/10.4171/RMI/1346},
  DOI = {10.4171/rmi/1346},
  number = {1},
  journal = {Revista Matemática Iberoamericana},
  publisher = {European Mathematical Society - EMS - Publishing House GmbH},
  author = {Chalmoukis,  Nikolaos and Daskalogiannis,  Vassilis},
  year = {2022},
  month = apr,
  pages = {321–340}
}

@article{Chen,
 author = {Chen, Jiale},
 title = {Some properties of the integration operators on the spaces {{\({F}(p, q, s)\)}}},
 fjournal = {Acta Mathematica Scientia. Series B. (English Edition)},
 journal = {Acta Math. Sci., Ser. B, Engl. Ed.},
 issn = {0252-9602},
 volume = {44},
 number = {1},
 pages = {173--188},
 year = {2024},
 language = {English},
 doi = {10.1007/s10473-024-0109-z},
 zbMATH = {7784078},
 Zbl = {1538.47075}
}

@article{DaskGal,
 author = {Daskalogiannis, Vassilis and Galanopoulos, Petros},
 title = {Semigroups of composition operators and integral operators in {BMOA}-type spaces},
 fjournal = {Complex Analysis and Operator Theory},
 journal = {Complex Anal. Oper. Theory},
 issn = {1661-8254},
 volume = {15},
 number = {8},
 pages = {30},
 note = {Id/No 117},
 year = {2021},
 language = {English},
 doi = {10.1007/s11785-021-01168-6},
 zbMATH = {7438143},
 Zbl = {1491.47030}
}

@article{GalMerchanSiskakis,
 author = {Galanopoulos, Petros and Merch{\'a}n, Noel and Siskakis, Aristomenis G.},
 title = {Semigroups of composition operators in analytic {Morrey} spaces},
 fjournal = {Integral Equations and Operator Theory},
 journal = {Integral Equations Oper. Theory},
 issn = {0378-620X},
 volume = {92},
 number = {2},
 pages = {15},
 note = {Id/No 12},
 year = {2020},
 language = {English},
 doi = {10.1007/s00020-020-2568-5},
 zbMATH = {7181798},
 Zbl = {1525.30037}
}

@article{Gumenyuk,
 author = {Gumenyuk, Pavel},
 title = {Angular and unrestricted limits of one-parameter semigroups in the unit disk},
 fjournal = {Journal of Mathematical Analysis and Applications},
 journal = {J. Math. Anal. Appl.},
 issn = {0022-247X},
 volume = {417},
 number = {1},
 pages = {200--224},
 year = {2014},
 language = {English},
 doi = {10.1016/j.jmaa.2014.02.057},
 zbMATH = {6391265},
 Zbl = {1305.30015}
}

@article{OrtegaFabrega,
 author = {Ortega, Joaqu{\'{\i}}n M. and F{\`a}brega, Joan},
 title = {Pointwise multipliers and corona type decomposition in {{\(BMOA\)}}},
 fjournal = {Annales de l'Institut Fourier},
 journal = {Ann. Inst. Fourier},
 issn = {0373-0956},
 volume = {46},
 number = {1},
 pages = {111--137},
 year = {1996},
 language = {English},
 doi = {10.5802/aif.1509},
 url = {https://eudml.org/doc/75168},
 zbMATH = {853997},
 Zbl = {0840.32001}
}

@article{PauZhao,
 author = {Pau, Jordi and Zhao, Ruhan},
 title = {Carleson measures, {Riemann}-{Stieltjes} and multiplication operators on a general family of function spaces},
 fjournal = {Integral Equations and Operator Theory},
 journal = {Integral Equations Oper. Theory},
 issn = {0378-620X},
 volume = {78},
 number = {4},
 pages = {483--514},
 year = {2014},
 language = {English},
 doi = {10.1007/s00020-014-2124-2},
 url = {hdl.handle.net/2445/96750},
 zbMATH = {6351692},
 Zbl = {1325.47075}
}

@article{PerfektDual,
 author = {Perfekt, Karl-Mikael},
 title = {Duality and distance formulas in spaces defined by means of oscillation},
 fjournal = {Arkiv f{\"o}r Matematik},
 journal = {Ark. Mat.},
 issn = {0004-2080},
 volume = {51},
 number = {2},
 pages = {345--361},
 year = {2013},
 language = {English},
 doi = {10.1007/s11512-012-0175-7},
 zbMATH = {6203781},
 Zbl = {1283.46011}
}

@article{PerfektCompact,
 author = {Perfekt, Karl-Mikael},
 title = {Weak compactness of operators acting on o-{O} type spaces},
 fjournal = {Bulletin of the London Mathematical Society},
 journal = {Bull. Lond. Math. Soc.},
 issn = {0024-6093},
 volume = {47},
 number = {4},
 pages = {677--685},
 year = {2015},
 language = {English},
 doi = {10.1112/blms/bdv031},
 zbMATH = {6472446},
 Zbl = {1333.46018}
}

@article{SiskakisDirichlet,
 author = {Siskakis, Aristomenis G.},
 title = {Semigroups of composition operators on the {Dirichlet} space},
 fjournal = {Results in Mathematics},
 journal = {Result. Math.},
 issn = {1422-6383},
 volume = {30},
 number = {1-2},
 pages = {165--173},
 year = {1996},
 language = {English},
 doi = {10.1007/BF03322189},
 zbMATH = {933265},
 Zbl = {0865.47030}
}

@incollection{SiskakisZhao,
 author = {Siskakis, Aristomenis G. and Zhao, Ruhan},
 title = {A {Volterra} type operator on spaces of analytic functions},
 booktitle = {Function spaces. Proceedings of the 3rd conference, Edwardsville, IL, USA, May 19--23, 1998},
 isbn = {0-8218-0939-3},
 pages = {299--311},
 year = {1999},
 publisher = {Providence, RI: American Mathematical Society},
 language = {English},
 zbMATH = {1306959},
 Zbl = {0955.47029}
}

@book{zhao96,
 author = {Zhao, Ruhan},
 title = {On a general family of function spaces},
 fseries = {Annales Academi{{\ae}} Scientiarum Fennic{{\ae}}. Mathematica. Dissertationes},
 series = {Ann. Acad. Sci. Fenn., Math., Diss.},
 issn = {1239-6303},
 volume = {105},
 isbn = {951-41-0795-0},
 year = {1996},
 publisher = {Helsinki: Suomalainen Tiedeakatemia},
 language = {English},
 zbMATH = {896885},
 Zbl = {0851.30017}
}

@article{zhaosurvey,
 author = {Zhao, Ruhan},
 title = {On {{\({F}(p, q, s)\)}} spaces},
 fjournal = {Acta Mathematica Scientia. Series B. (English Edition)},
 journal = {Acta Math. Sci., Ser. B, Engl. Ed.},
 issn = {0252-9602},
 volume = {41},
 number = {6},
 pages = {1985--2020},
 year = {2021},
 language = {English},
 doi = {10.1007/s10473-021-0613-3},
 zbMATH = {7560117},
 Zbl = {1513.30206}
}

@book{zhu,
 author = {Zhu, Kehe},
 title = {Operator theory in function spaces},
 fseries = {Pure and Applied Mathematics, Marcel Dekker},
 series = {Pure Appl. Math., Marcel Dekker},
 volume = {139},
 isbn = {0-8247-8411-1},
 year = {1990},
 publisher = {New York etc.: Marcel Dekker, Inc.},
 language = {English},
 zbMATH = {47946},
 Zbl = {0706.47019}
}

@article{Berkson1978,
 author = {Berkson, Earl and Porta, Horacio},
 title = {Semigroups of analytic functions and composition operators},
 fjournal = {Michigan Mathematical Journal},
 journal = {Mich. Math. J.},
 issn = {0026-2285},
 volume = {25},
 pages = {101--115},
 year = {1978},
 language = {English},
 doi = {10.1307/mmj/1029002009},
 zbMATH = {3594120},
 Zbl = {0382.47017}
}

@incollection{Siskakis1998,
 author = {Siskakis, Aristomenis G.},
 title = {Semigroups of composition operators on spaces of analytic functions, a review},
 booktitle = {Studies on composition operators. Proceedings of the Rocky Mountain Mathematics Consortium, Laramie, WY, USA, July 8--19, 1996},
 isbn = {0-8218-0768-4},
 pages = {229--252},
 year = {1998},
 publisher = {Providence, RI: American Mathematical Society},
 language = {English},
 zbMATH = {1107617},
 Zbl = {0904.47030}
}

@article{Blasco2008,
 author = {Blasco, Oscar and Contreras, Manuel D. and D{\'{\i}}az-Madrigal, Santiago and Mart{\'{\i}}nez, Josep and Siskakis, Aristomenis G.},
 title = {Semigroups of composition operators in {BMOA} and the extension of a theorem of {Sarason}},
 fjournal = {Integral Equations and Operator Theory},
 journal = {Integral Equations Oper. Theory},
 issn = {0378-620X},
 volume = {61},
 number = {1},
 pages = {45--62},
 year = {2008},
 language = {English},
 doi = {10.1007/s00020-008-1568-7},
 zbMATH = {5294664},
 Zbl = {1171.47034}
}

@article{Wu2021,
 author = {Wu, Fanglei and Wulan, Hasi},
 title = {Semigroups of composition operators on {{\(\mathcal{Q}_p\)}} spaces},
 fjournal = {Journal of Mathematical Analysis and Applications},
 journal = {J. Math. Anal. Appl.},
 issn = {0022-247X},
 volume = {496},
 number = {2},
 pages = {15},
 note = {Id/No 124845},
 year = {2021},
 language = {English},
 doi = {10.1016/j.jmaa.2020.124845},
 zbMATH = {7316452},
 Zbl = {1480.47064}
}

@article{Sarason1975,
 author = {Sarason, Donald},
 title = {Functions of vanishing mean oscillation},
 fjournal = {Transactions of the American Mathematical Society},
 journal = {Trans. Am. Math. Soc.},
 issn = {0002-9947},
 volume = {207},
 pages = {391--405},
 year = {1975},
 language = {English},
 doi = {10.2307/1997184},
 zbMATH = {3500403},
 Zbl = {0319.42006}
}

@article{Aluaskari1997,
 author = {Aulaskari, Rauno and Lappan, Peter and Xiao, Jie and Zhao, Ruhan},
 title = {On {{\(\alpha\)}}-{Bloch} spaces and multipliers of {Dirichlet} spaces},
 fjournal = {Journal of Mathematical Analysis and Applications},
 journal = {J. Math. Anal. Appl.},
 issn = {0022-247X},
 volume = {209},
 number = {1},
 pages = {103--121},
 year = {1997},
 language = {English},
 doi = {10.1006/jmaa.1997.5345},
 zbMATH = {1031556},
 Zbl = {0892.30030}
}

@article {Pommerenke1978,
AUTHOR = {Pommerenke, Christian},
TITLE = {On univalent functions, {B}loch functions and {VMOA}},
JOURNAL = {Math. Ann.},
FJOURNAL = {Mathematische Annalen},
VOLUME = {236},
YEAR = {1978},
NUMBER = {3},
PAGES = {199--208},
ISSN = {0025-5831,1432-1807},
MRCLASS = {30A32 (30A36)},
MRNUMBER = {492206},
MRREVIEWER = {L.\ A.\ Aksent\cprime ev},
DOI = {10.1007/BF01351365},
URL = {https://doi.org/10.1007/BF01351365},
}

@article {GirPel,
AUTHOR = {Girela, Daniel and Pel\'aez, Jos\'e{} \'Angel},
TITLE = {Growth properties and sequences of zeros of analytic functions in spaces of {D}irichlet type},
JOURNAL = {J. Aust. Math. Soc.},
FJOURNAL = {Journal of the Australian Mathematical Society},
VOLUME = {80},
YEAR = {2006},
NUMBER = {3},
PAGES = {397--418},
ISSN = {1446-7887,1446-8107},
MRCLASS = {30D35 (30D55 46E15)},
MRNUMBER = {2236047},
MRREVIEWER = {L.\ R.\ Sons},
DOI = {10.1017/S1446788700014105},
URL = {https://doi.org/10.1017/S1446788700014105},
}

\end{document}